\documentclass[10pt]{amsart}
\usepackage{amsmath,epsf}
\usepackage{xr}

\usepackage{amssymb,stmaryrd,enumitem}
\usepackage{tikz} \usetikzlibrary{matrix,arrows,decorations.pathmorphing}

\subjclass[2000]{Primary 03C60, 11G25. Secondary 14G10, 14G15.}
\keywords{difference scheme, Galois stratification, Galois formula, Frobenius automorphism, ACFA}
\title{Twisted Galois stratification}
\date{\today}

\author{Ivan Toma{\v s}i{\'c}}
\address{Ivan Toma{\v s}i{\'c}\\
         School of Mathematical Sciences\\
  	Queen Mary University of London\\
         London, E1 4NS\\
        United Kingdom}
\email{i.tomasic@qmul.ac.uk}
\theoremstyle{plain}
\newtheorem{theorem}{Theorem}[section]
\newtheorem{corollary}[theorem]{Corollary}
\newtheorem{proposition}[theorem]{Proposition}
\newtheorem{lemma}[theorem]{Lemma}

\theoremstyle{definition}
\newtheorem{definition}[theorem]{Definition}

\theoremstyle{remark}
\newtheorem{remark}[theorem]{Remark}

\newtheorem{conjecture}[theorem]{Conjecture}

\providecommand{\OO}{\mathcal{O}}

\providecommand{\Z}{\mathbb{Z}}
\providecommand{\A}{\mathfrak{a}}

\providecommand{\p}{\mathfrak{p}}
\providecommand{\q}{\mathfrak{q}}

\providecommand{\cF}{\mathcal{F}}

\providecommand{\cA}{\mathcal{A}}
\providecommand{\cB}{\mathcal{B}}
\providecommand{\cC}{\mathcal{C}}

\providecommand{\kk}{\mathbf{k}}

\providecommand{\spec}{{\rm Spec}}

\providecommand{\F}{\mathbb{F}}
\providecommand{\Af}{\mathbb{A}}

\providecommand{\Gal}{\text{\rm Gal}}
\providecommand{\ar}{\mathop{\rm ar}\nolimits}
\providecommand{\con}{\mathop{\rm con}\nolimits}

\providecommand{\Hom}{\text{\rm Hom}}
\providecommand{\diff}{\text{\it Diff}}

\def\mathrlap{\mathpalette\mathrlapinternal}
\def\mathrlapinternal#1#2{%
           \rlap{$\mathsurround=0pt#1{#2}$}}

\providecommand{\ztild}[1]{\rlap{$\smash{\tilde{\phantom{#1}}}$}\rlap{$\mathring{\phantom{#1}\kern1.1ex}$}#1\kern.1ex}

\providecommand{\lexp}[2]{{\vphantom{#2}}^{#1}{\kern-.1ex#2}}

\providecommand{\acirc}[1]{\phantom{a}\llap{$\scriptstyle#1$}
\kern.01ex\lower.75ex\hbox{$\smash{\mathring{}}$}}

\providecommand{\lzexp}[3]{{\vphantom{#2}}^{\lower0.0ex\hbox{\smash{$\acirc{#1}$}}}\kern-.1ex #2}

\providecommand{\lrexp}[3]{{\vphantom{#2}}^{#1}{\kern-.1ex#2^#3}}
\begin{document}

\begin{abstract} 
We prove a direct image theorem stating that the direct image of a Galois
formula  by a morphism of difference schemes 
is equivalent to a Galois formula over fields with powers of Frobenius. 
As a consequence,
we obtain an \emph{effective} quantifier elimination procedure and a
 precise algebraic-geometric description of definable sets over fields with Frobenii in terms of
twisted Galois formulae associated with finite Galois covers of difference schemes.
\end{abstract}
\maketitle

\tableofcontents

\section{Introduction}

Galois stratification,  originally developed through work of Fried, Haran, Jarden and Sacerdote
(\cite{fried-jarden}, \cite{fried-sacer}, \cite{FHJ}),
provides an explicit arithmetic-geometric description of definable sets over finite fields
in terms of Galois formulae associated to Galois coverings of algebraic varieties. 
When compared to the
earlier work of Ax \cite{Ax}, made more explicit by Kiefe \cite{kiefe}, which showed that every formula in the language of rings
is equivalent to a formula with a single (bounded) existential quantifier, the
fundamental achievement of the Galois stratification was the \emph{effective}
(in fact \emph{primitive recursive}) nature of its quantifier elimination procedure.  
Moreover, the precise description of formulae in terms of Galois covers was
particularly well-suited for applications of geometric and number-theoretic nature, 
for example
in Fried's work on Davenport's problem \cite{fried-davenp}. In our opinion, the most impressive application was in the
work of Denef and Loeser on arithmetic motivic integration in \cite{DL}. They assign a Chow motive
to a Galois formula, thus extending the consideration of algebraic-geometric invariants of algebraic
varieties to arbitrary first-order formulae. 
 
We develop the theory of \emph{twisted Galois stratification} in order to describe first-order
definable sets in the language of \emph{difference rings} over algebraic closures
of finite fields equipped with powers
of the Frobenius automorphism. 
A (normal) \emph{Galois stratification} on a difference scheme $(X,\sigma)$ is a datum
$$
\cA=\langle X, Z_i/X_i,C_i\ |\ i\in I\rangle,
$$
where $X_i$, $i\in I$ is a partition of $X$ into finitely many normal 
locally closed difference
subschemes of $X$, each $(Z_i,\Sigma_i)\to (X_i,\sigma)$ is a Galois
covering with some group $(G_i,\tilde{\Sigma})$ and $C_i$ is a conjugacy
domain in $\Sigma_i$, with all these notions precisely defined in \cite{ive-tch}. 
The \emph{Galois formula} associated with $\cA$
is the realisation subfunctor $\tilde{\cA}$ of $X$ defined by the assignment
$$
\tilde{\cA}(F,\varphi)=\bigcup_{i\in I}\{x\in X_i(F,\varphi):\varphi_x\subseteq C_i\}
\subseteq X(F,\varphi),
$$
where $(F,\varphi)$ is an algebraically closed difference field and the 
conjugacy class
$\varphi_x\subseteq\Sigma$ is the local $\varphi$-substitution at $x$, as defined
in \cite[Section~\ref{sect:galois}]{ive-tch}. 
Our principal result in its algebraic-geometric incarnation is
 the following \emph{direct image theorem}, stating that
a direct image of a Galois formula by a morphism of finite transformal type
is equivalent to a Galois formula over fields with Frobenii. Equivalently,
the class of Galois formulae over fields with Frobenii is closed under taking
direct images by morphisms of finite transformal type
(a precise statement is \ref{eximageisgal}).
\begin{theorem}\label{main-dirim}
Let $f:(X,\sigma)\to(Y,\sigma)$ be a morphism of finite transformal type
(over a suitable base), and
let $\cA$ be a Galois stratification on $X$. We can effectively compute a Galois stratification
$\cB$ on $Y$ such that for all (suitable) $(\bar{\F}_p,\varphi)$ with a high enough
power of Frobenius $\varphi$,
$$
f(\tilde{\cA}(\bar{\F}_p,\varphi))=\tilde{\cB}(\bar{\F}_p,\varphi).
$$ 
\end{theorem}
A model-theoretic restatement of the above theorem is that fields with Frobenii
allow \emph{quantifier elimination} in the language of Galois formulae. In other words,
any definable set over fields with powers of
Frobenius can be described by a  Galois formula (a precise
statement is \ref{qe-ffrob}).
\begin{theorem}\label{main-QE}
Let $\theta(x_1,\ldots,x_n)$ be a first-order formula in the language of
difference rings (with suitable parameters). 
We can effectively compute a Galois stratification $\cA$ of the 
difference affine $n$-space such that for all (suitable) $(\bar{\F}_p,\varphi)$
with a high enough power of Frobenius $\varphi$, 
$$
\theta(\bar{\F}_p,\varphi)=\tilde{\cA}(\bar{\F}_p,\varphi).
$$
Conversely, every Galois formula is equivalent to a first-order formula in the
language of difference rings over fields with Frobenii.
\end{theorem}

Historically speaking, the comparison of our result to the known 
model-theoretic quantifier
elimination result found by Macintyre \cite{angus} and greatly refined in modern
terms by
Chatzidakis and Hrushovski \cite{zoe-udi}, is
parallel to the relation between the work of Fried-Sacerdote and the work of Ax
mentioned above. Macintyre and Chatzidakis-Hrushovski show that any formula 
$\theta(x_1,\ldots,x_n)$ in the
language of difference rings is equivalent, modulo the theory ACFA of 
existentially closed difference fields, to a Boolean combination of formulae
of the form $$\exists y\,\,\psi(y;x_1,\ldots,x_n),$$ where $\psi$ is quantifier free,
and $\psi(y;x_1,\ldots,x_n)$ implies that $y$ satisfies a nonzero polynomial
whose coefficients are $\sigma$-polynomials in $x_1,\ldots,x_n$, i.e., the
single existential quantifier is bounded. 
Their proof uses the compactness theorem and, although recursive, 
their quantifier elimination is far from being primitive recursive or effective
in a suitable sense of the word. 

The main achievement of our paper is
the \emph{effectivity} of our quantifier elimination procedure, the proof of the
direct image theorem being fundamentally algorithmic and algebraic-geometric
in nature. We show that our quantifier elimination and the decision procedure for
fields with Frobenii are \emph{$\dag$-primitive recursive,} i.e., 
primitive-recursive reducible to basic operations in difference algebraic geometry,
as detailed in Section~\ref{s:effective}. Although primitive-recursive algorithms
are not known at the moment for some of the most elementary constructions in
difference algebra, we strongly believe that in the near future our procedures
will be shown to be primitive recursive.
Needless to say, while it may be possible to start with the model-theoretic quantifier elimination and deduce the
precise form of \ref{main-QE} (in fact \ref{stratACFA}),  taking this route would be
 missing the point.

The statement \ref{main-QE}  is over fields with
Frobenii and that is why we must refer to the present author's Chebotarev Lemma
\cite[\ref{twisted-cebotarev}]{ive-tch} which uses the difficult paper \cite{udi} on twisted
Lang-Weil estimates, proving the earlier conjecture of \cite{angus} that
ACFA is the elementary theory of fields with Frobenii. This is the only use
of the main result of \cite{udi} in this paper, and the remaining references to \cite{udi}
are mostly foundational lemmas.
However, our Galois stratification procedure 
works over existentially closed difference fields unconditionally, 
\emph{without} the use of \cite{udi}, see \ref{stratACFA}.

One of the biggest challenges was 
the correct formulation of the result and even a suitable definition of a 
\emph{Galois cover}, which already requires the full power of the
theory of \emph{generalised difference schemes} developed in \cite{ive-tch},
since the  category of \emph{strict} difference schemes has no 
reasonable Galois actions, coverings or quotients. One clear advantage
of the description of the definable sets in terms of (twisted) Galois stratifications
is our ability to reduce considerations regarding points on definable sets to 
calculations of various character sums, as
expounded in \cite{ive-tch}. Since the style of our proof is reminiscent of many a
direct image theorem from algebraic geometry, our results should appeal
to algebraic geometers and number theorists and we expect more diophantine
applications to follow.

Our approach to the proof of \ref{main-dirim} (in fact of \ref{eximageisgal}) 
is  more geometric and
conceptual than those of \cite{fried-jarden}, \cite{fried-sacer}, \cite{FHJ} in 
the classical case.
The proof from \cite{nicaise} in the algebraic case uses the theory of the
\'etale fundamental group in a rather sophisticated
way,  which is not available in the difference scenario.
However, by performing 
a `baby' Stein factorisation at the start of our procedure,
the only remnant of that theory is
 the short exact sequence for the \'etale fundamental group, in which
 case we can `manually' keep track of what happens at the level of
 finite Galois covers. From this point of view, even if we were to
 eliminate all the difference language, our line of proof  would
 still yield an essentially new proof in the classical case.
Here, on the other hand, we must treat several genuinely new difference phenomena which do not arise in the algebraic case. Key ingredients include
Babbitt's decomposition theorem \ref{babb-sch} and our
Chebotarev lemma \cite[\ref{twisted-cebotarev}]{ive-tch}. 

Although the foundation of the theory of generalised difference schemes has been
laid in \cite{ive-tch},
 for the purposes
of this paper we must develop the
framework even further in Section~\ref{sect:bi-fib}.

In the course of the proof, we use local properties of difference schemes previously unknown
 in difference algebraic geometry, developed in Section~\ref{sec:localstudy}.
 It must be emphasised that our theory is almost orthogonal to the various
 notions of smoothness that appear in G.~Giabicani's thesis, see \ref{remgiab}.

En route to the main theorem, we encounter another merit of working in the context of
generalised difference schemes, a difference version of Chevalley's theorem
\ref{chevalley}, 
which gives a sufficient condition for the 
image of a morphism of difference schemes
 of finite transformal type to contain a dense open set. Wibmer 
 gives a similar result by generalising difference algebra in a slightly different 
 direction \cite{wibmer}.

 The author would like to thank Michael Fried, Angus Macintyre and
 Thomas Scanlon for fruitful discussions on the topic of this paper, and to
 Zoe Chatzidakis for pointing out the importance of Babbitt's decomposition
 to him many years ago.

\section{Local study of difference schemes and their morphisms}\label{sec:localstudy}

The foundation of the theory of  generalised
difference schemes started has been laid in \cite{ive-tch}. The familiarity
with this work is crucial and we freely use the concepts defined there, 
although the reader acquainted with \cite{udi}
can follow the subsequent developments that refer to ordinary difference schemes.

\subsection{Difference schemes vs.\ pro-algebraic varieties}\label{s:diff-pro}

One of the most important ideas in the study of difference algebraic geometry
was the realisation that there is a translation mechanism between the language
of difference schemes and that of algebraic correspondences, or, more generally,
systems of prolongations associated with a difference scheme.
 
We would like to be able to reduce the study of certain local properties of difference schemes
to the study of known properties of algebraic schemes through systems of prolongations. 
In order to achieve this goal, we must be able to
speak about \emph{difference subvarieties} of ordinary algebraic
varieties, which is achieved by defining a 
\emph{difference scheme associated to a scheme}. 
 
\begin{proposition}[\cite{udi}]
Let $(R,\sigma)$ be a difference ring. The forgetful functor
from the category of difference $(R,\sigma)$-schemes to the category of locally $R$-ringed spaces
has a right adjoint $[\sigma]_R$, i.e., for every $R$-scheme $X$ we have a morphism
$[\sigma]_RX\to X$ inducing the functorial isomorphism
$$
\Hom_R(Z,X)=\Hom_{(R,\sigma)}(Z,[\sigma]_RX),
$$
for every $(R,\sigma)$-difference scheme $(Z,\sigma)$, where the morphisms on the left
are the morphisms of locally $R$-ringed spaces.   
\end{proposition}

Suppose now that $X$ is an $R$-scheme and $(Z,\sigma)$ is a closed 
$(R,\sigma)$-difference subscheme of $[\sigma]_RX$. For ease of notation, let us write 
$S=\spec(R)$. Writing $X^{\sigma^n}$ for
$X\times_SS$ where the morphism $S\to S$ is $\sigma^n$, it is clear that the $\sigma^n$-linear morphism 
$\sigma^n:[\sigma]_RX\to[\sigma]_RX$ defines an $R$-morphism $[\sigma]_RX\to X^{\sigma^n}$ 
and thus we deduce a morphism 
$$Z\hookrightarrow [\sigma]_RX\to X\times X^\sigma\times\cdots\times X^{\sigma^n}=:X[n].$$ We denote the scheme-theoretic image of this map by $Z[n]$, obtaining a closed $R$-subscheme
$Z[n]\hookrightarrow X[n]$ for every $n$. Although the projective limit $Z[\infty]$ of the $Z[n]$
can be viewed as a scheme, we will find it most illuminating to view it as a pro-(scheme of finite type).

The maps $X^{\sigma^{n+1}}\to X^{\sigma^{n}}$ induce the maps $\sigma:X[n+1]\to X[n]$
and thus $X[\infty]$ is equipped with an endomorphism induced by $\sigma$. 
In particular, this defines an isomorphism %
$\sigma:\prod_{n\geq 1}X^{\sigma^n}\to\prod_{n\geq 0}X^{\sigma^n}=X[\infty]$ 
We also have the projection $r:\prod_{n\geq 0}X^{\sigma^n}\to\prod_{n\geq 1}X^{\sigma^n}$.
\begin{proposition}[\cite{udi}]
An $R$-subscheme $Y$ of $X[\infty]$ is of the form $Z[\infty]$ for some 
difference subscheme $(Z,\sigma)$ of $[\sigma]_RX$ if and only if
$Y$ contains $rY^{\sigma}%
$.   %
\end{proposition}

We say that $Z$ is \emph{weakly Zariski dense} in $X$ if $Z[0]=X$. Note that it can happen
that $Z$ is weakly Zariski dense in $X$ but the set of points of $Z$ is not Zariski dense in  $X$.

Let us now start with a difference scheme $(X,\sigma)$ of finite $\sigma$-type 
over $(R,\sigma)$ and build a system of `prolongations' of $X$ in which $X$ is
weakly Zariski dense by construction. We shall describe the procedure for an affine 
difference scheme $(X,\sigma)=\spec^\sigma(A)$, where $A=R[a]_\sigma$ is an 
$(R,\sigma)$-algebra of finite $\sigma$-type, generated by a tuple $a\in A$.

If we write $A_n:=R[a,\sigma a,\ldots,\sigma^n a]$, we have inclusions $A_n\hookrightarrow A_{n+1}$
and maps $\sigma_n:A_n\to A_{n+1}$ induced by $\sigma$, so that $(A,\sigma)$ is the direct limit of
the $A_n$ and the $\sigma_n$. We obtain the following diagram for $X_n=\spec(A_n)$.
$$
 \begin{tikzpicture}
[cross line/.style={preaction={draw=white, -,
line width=4pt}}]
\matrix(m)[matrix of math nodes, row sep=.8cm, column sep=.8cm, text height=1.5ex, text depth=0.25ex]
{|(0m1)|{S}& |(00)|{X_0} & |(01)|{X_1} & |(02)|{X_2}& |(03)|{\cdots}& |(04)|{}\\
|(1m1)|{S}&|(10)|{X_0^\sigma}&|(11)|{X_1^\sigma}&|(12)|{X_2^\sigma}&|(13)|{\cdots}&|(14)|{}\\
         &|(20)|{X_0} & |(21)|{X_1}&|(22)|{X_2}&|(23)|{X_3}&|(24)|{\cdots}\\};%
\path[->,font=\scriptsize,>=to, thin]
(1m1) edge node[right]{$\sigma$} (0m1)
(10) edge (00)(11) edge (01)(12) edge (02) %
(00) edge (0m1) (01) edge (00) (02) edge (01) (03) edge (02) %
(10)edge(1m1)(11)edge(10)(12)edge(11)(13)edge(12)%
                               (21)edge(20)(22)edge(21)(23)edge(22)(24)edge(23)
 (20)edge(1m1) (21)edge(10) (22)edge(11) (23)edge(12) %
(21)edge[cross line] node[pos=0.80,right]{$\sigma_0$}(00)%
(22)edge[cross line] node[pos=0.80,right]{$\sigma_1$}(01)%
(23)edge[cross line]node[pos=0.80,right]{$\sigma_2$}(02)  %
;
\end{tikzpicture}
$$ 

By construction, we have closed immersions $X_1\hookrightarrow X_0\times_SX_0^\sigma$
and $X_{n+1}\hookrightarrow X_n\times_{X_{n-1}^\sigma}X_n^\sigma$ for $n\geq 1$, and
we conclude that we have written $(X,\sigma)$ as a weakly Zariski dense difference 
subscheme of $X_0$.

\begin{lemma}[Preparation Lemma]\label{preprl}
With above notation, if $A$ and $R$ are algebraically integral and 
all $A_n\to A_{n+1}$ are separable, by $\sigma$-localising
$A$ and $R$ we can arrange that morphisms 
$$X_{n+1}\rightarrow X_n\times_{X_{n-1}^\sigma}X_n^\sigma
\to X_n\times_{X_{n-1}}X_n$$
 are isomorphisms for $n\geq 1$ and that $X$ is a Zariski dense
difference subscheme of $X_0$.
\end{lemma}
\begin{proof}
Let $K$ be the fraction field of $R$.
By combining the statements 5.2.10, 5.2.11, 5.2.12 from \cite{levin}, modulo
a $\sigma$-localisation of $R$, we can
find a new tuple of generators $a=bc$ so that, writing $\sigma^i(a)=a_i=b_ic_i$
and $L_n=K(a_0,\ldots,a_n)$ for the fraction field of $A_n$, we have
that  for $n\geq 1$,
\begin{enumerate}
\item\label{prvi} $b_n$ is algebraically
independent over $L_{n-1}$, and 
\item\label{drugi}
$[L_n:L_{n-1}(b_n)]=
[L_{n+1}:L_n(b_{n+1})]$.
\end{enumerate}
Given a diagram
$$
 \begin{tikzpicture} 
\matrix(m)[matrix of math nodes, row sep=2.6em, column sep=2.8em, text height=1.5ex, text depth=0.25ex]
 {k & B\\
   C & K\\}; 
\path[->,font=\scriptsize,>=to, thin]
(m-2-1) edge (m-2-2) 
(m-1-1) edge node[left] {${\sigma}$} (m-2-1)
(m-1-2) edge node[right] {${\sigma}$} (m-2-2)
(m-1-1) edge (m-1-2);
\end{tikzpicture}
$$
of sub-$k$-algebras $B$, $C$ of a difference field $(K,\sigma)$, we shall say that 
$B$ is $\sigma$-linearly disjoint from $C$ over $k$ if
whenever $\{\beta_1,\ldots,\beta_r\}\subseteq B$ is linearly independent over $k$,
then $\{\sigma(\beta_1),\ldots,\sigma(\beta_r)\}$ is linearly independent over $C$. 
This is equivalent to the injectivity of the natural map 
$B\otimes_kC\to\sigma(B)C$.

Using (\ref{prvi}), we see that $L_{n-1}(b_{n})$ is $\sigma$-linearly disjoint from $L_{n}$
over $L_{n-1}$, and using (\ref{drugi}) we deduce that $L_n$ is $\sigma$-linearly disjoint
from $L_n(b_{n+1})$ over $L_{n-1}(b_n)$. 
 By transitivity of $\sigma$-linear disjointness,
it follows from the diagram
$$
 \begin{tikzpicture}
[cross line/.style={preaction={draw=white, -,
line width=4pt}}]
\matrix(m)[matrix of math nodes, row sep=1cm, column sep=.5cm, text height=1.5ex, text depth=0.25ex]
{|(0m1)|{\cdots}& |(00)|{L_{n-1}} & |(01)|{L_{n-1}(b_n)} & |(02)|{L_n}& |(03)|{
\cdots\phantom{aaaa}}& |(04)|{}&\\
|(1m1)|{}&|(10)|{}&|(11)|{
\phantom{aaaa}\cdots}&|(12)|{L_n}&|(13)|{L_n(b_{n+1})}&|(14)|{L_{n+1}}&|(15)|{\cdots} \\
};
\path[->,font=\scriptsize,>=to, thin]
(0m1)edge(00) (00)edge(01) (01)edge(02) (02)edge(03) 
(11.east)edge(12) (12)edge(13) (13) edge (14)  (14)edge(15)
(00)edge node[right=4pt,pos=0.4]{$\sigma$}(12) (01)edge node[right=4pt,pos=0.4]{$\sigma$}(13) (02)edge node[right=4pt,pos=0.4]{$\sigma$}(14)       
;
\end{tikzpicture}
$$ 
 that $L_n$ is $\sigma$-linearly disjoint from $L_n$ over $L_{n-1}$ for all $n\geq 1$.

Now, using generic freenes \cite[Lemme 6.9.2]{EGAIV(2)},  by 
$\sigma$-localising $A$ (by an element of $A_0$)
we may assume that $\pi_{10}:A_0\to A_1$ and $\sigma_0:A_0\to A_1$
are free (i.e., $A_1$ is a free $A_0$-module both via $\pi_{10}$ and $\sigma_0$),
 so $A_1\otimes_{A_0}A_1$ is a free $A_0$ module. Thus
the natural morphism
$$
(A_1\otimes_{A_0}A_1)\to L_0\otimes_{A_0}(A_1\otimes_{A_0}A_1)$$
is injective, the kernel in general being the $A_0$-torsion of $A_1\otimes_{A_0}A_1$,
which we thoughtfully made trivial. Moreover, by linear disjointness guaranteed
by the construction, 
\begin{multline*}
L_0\otimes_{A_0}(A_1\otimes_{A_0}A_1)=
(L_0\otimes_{A_0}A_1)\otimes_{L_0}(L_0\otimes_{A_0}A_1)\\
=L_0[A_1]\otimes_{L_0}L_0[A_1]\to L_0[A_1,\sigma(A_1)]=L_0[A_2]
\end{multline*}
is injective. We conclude that
$A_1\otimes_{A_{0}} A_1\to A_{2}$ is injective and thus bijective by the construction, and that both $\pi_{21}:A_1\to A_2$ and $\sigma_1:A_1\to A_2$ are free.
This is all we need to proceed by induction and prove that
all  $$A_n\otimes_{A_{n-1}} A_n\to A_n\otimes_{A_{n-1}^\sigma}A_n^\sigma\to A_{n+1}$$ are isomorphisms.
Note, if we are happy to finish with the associated
 morphisms being just closed immersions
which are generically isomorphisms, we can skip the `generic freeness' step and
we do not need to localise $A$ but only $R$.
\end{proof}

\begin{definition}
Let $P$ be a property of scheme morphisms of finite type.
Consider the following permanence properties of $P$:
\begin{enumerate}
\item (Composite). A composite of morphisms with property $P$ has property $P$.
\item (Base change). If $X\to Y$ has $P$, and $Z\to Y$ is arbitrary, then $X\times_YZ\to Z$ has $P$.
\item (Open embedding). If $X\to Y$ has $P$, and $U\hookrightarrow X$, then $U\to Y$ has $P$.
\item\label{trgt} (Genericity in the target). If $f:X\to Y$ (with $Y$ integral) is generically $P$, there is a localisation $Y'$ of
$Y$ such that $f\restriction f^{-1}(Y')$ is $P$.
\item[(\ref{trgt}')] (Genericity in the source). If $f:X\to Y$ (with $Y$ integral) is generically $P$, there is a localisation $X'$ of
$X$ and $Y'$ of $Y$ such that $f\restriction X'\cap f^{-1}(Y')$ is $P$.
\end{enumerate}
We say that $P$ is \emph{hereditary} if it has properties (1)--(3).
It is \emph{hereditarily generic in the target
(resp.~source)}, if it is hereditary with property (\ref{trgt}) (resp.~(\ref{trgt}')). 
The property  $P$ is \emph{strongly hereditary} if in addition,
\begin{enumerate}
\item[(5)] (SH). If $g\circ f$ and $g$ have $P$, then $f$ has $P$.
\end{enumerate}
\end{definition}

\begin{definition}
\begin{enumerate}
\item Let $(X,\sigma)$ be an $(R,\sigma)$-difference scheme of finite $\sigma$-type.
We say that $(X,\sigma)$ has the property $\sigma$-$P$, if there exists a prolongation
sequence $X_n$ as above for $X$ such that all the structure maps $X_n\to\spec(R)$
have the property $P$.

\item Let $P$ be a property of scheme morphisms of finite type. Let $f:(X,\sigma)\to (Y,\sigma)$ be a morphism of finite $\sigma$-type.
We say that $f$ has the property $\sigma$-$P$, if for every open affine $V=Spec^\sigma(R)$ in $Y$, the scheme $f^{-1}(V)$ has the property $\sigma$-$P$.

\item Let $P$ be a property of morphisms of schemes of finite type. Let $f:(X,\sigma)\to (Y,\sigma)$ be a morphism of schemes of finite $\sigma$-type (over some common base).
We say that $f$ has the property $\sigma$-$P$, if there exists a prolongation
sequence $f_n:X_n\to Y_n$ for $f$ such that all the maps $f_n$ have the property $P$.
\end{enumerate}
\end{definition}

\begin{remark}Suppose $P$ is strongly hereditary.
\begin{enumerate} 
\item If  $(X,\sigma)$ is $\sigma$-$P$, then all the
connecting morphisms $X_{n+1}\to X_n$ have the property $P$.
\item 
If a morphism  $f:(X,\sigma)\to (Y,\sigma)$ of schemes of finite $\sigma$-type (over some common base)  is $\sigma$-$P$, and $(Y,\sigma)$ is $\sigma$-$P$, then $(X,\sigma)$ is $\sigma$-$P$.
\end{enumerate}
\end{remark}

\begin{proposition}\label{sigma-generic}
\begin{enumerate}
\item\label{frsthf}
Let $P$ be a property of scheme morphisms of finite type which is hereditarily
generic in the source/target.
 Then the
property $\sigma$-$P$ is $\sigma$-generic in the source. 
In other words, if $f:(X,\sigma)\to (Y,\sigma)$
is a morphism of finite $\sigma$-type which is generically $\sigma$-$P$, then
there exists a $\sigma$-localisation $X'$ of $X$ and $Y'$ of $Y$ 
such that $f\restriction X'$ is 
$\sigma$-$P$ above $Y'$, i.e., $f\restriction X'\cap f^{-1}(Y')$ is $\sigma$-$P$.

\item\label{scnd} The same statements apply when $P$ is a (target/source) 
hereditarily generic property of morphisms of schemes of finite
type and $f:(X,\sigma)\to (Y,\sigma)$ is a morphism of schemes of finite $\sigma$-type.
\end{enumerate}
\end{proposition}
\begin{proof}
Let us prove (\ref{scnd}), the proof of (\ref{frsthf}) being strictly easier. 
We shall assume the reader has constructed, upon a $\sigma$-localisation of the source, the relevant diagram of prolongations
for $f_n:X_n\to Y_n$ using the Preparation Lemma~\ref{preprl}.
In the case of genericity in the target, by using (G), modulo a 
$\sigma$-localisation of $Y$ we can assume
that $X_0\to Y_0$ has $P$. Using (G) and (O), by $\sigma$-localising $X$ by an
element of $X_0$ we can assume that $X_1\to X_0$ also has $P$.

In the case of genericity in the source, using (G), by a $\sigma$-localisation of $X$
and $Y$
we can assume that $X_0\to Y_0$ has $P$. Using (G) again, we need to
$\sigma$-localise $X$ further to make $X_1\to X_0$ have the property $P$.
Using (O), the new $X_0\to Y_0$ still has $P$, but we lose the exact 
$\sigma$-generation in terms of fibre products to the extent that
$X_{n+1}\to X_n\times_{X_{n-1}^\sigma}X_n^\sigma$ are no longer
isomorphisms for $n\geq 1$, but only open immersions.

We proceed by induction. Assuming that $X_{n-1}\to Y_{n-1}$ and $X_n\to X_{n-1}$
have $P$, using (C), we get that $X_n\to Y_n$ has $P$. Moreover,
using (BC), we obtain that $X_n\times_{X_{n-1}}X_n\to X_n$ has $P$. 
By (O) and the fact that 
$X_{n+1}\hookrightarrow X_n\times_{X_{n-1}}X_n$ for $n\geq 1$, 
we can deduce that $X_{n+1}\to X_n$ also has $P$, which keeps the induction going.

Let us note that, in case of a property strongly hereditarily generic in the target, 
if the
preparation lemma could be improved so that we need only localise the base,
then we could prove that $\sigma$-$P$ is $\sigma$-generic in the target.
\end{proof}

\begin{corollary}\label{loc-si-smooth}
Let $f:(X,\sigma)\to (Y,\sigma)$ be a morphism of finite $\sigma$-type.
\begin{enumerate}
\item If $f$ is separable then there is a $\sigma$-localisation $X'$ of $X$
and $Y'$ of $Y$ such that $f\restriction X'\cap f^{-1}(Y')$ is $\sigma$-smooth.
\item If $f$ is separable algebraic, then there is a $\sigma$-localisation
$X'$ of $X$ and $Y'$ of $Y$ such that $f\restriction X'\cap f^{-1}(Y')$ is 
$\sigma$-\'etale.
\end{enumerate}
\end{corollary}

\begin{corollary}\label{loc-si-integralfibres}
Let $f:(X,\sigma)\to (Y,\sigma)$ be a morphism of finite $\sigma$-type
whose generic fibre is geometrically integral. Then there is a $\sigma$-localisation $Y'$ of $Y$
such that $f\restriction f^{-1}(Y')$ has geometrically integral fibres.
\end{corollary}
\begin{proof}
If $f$ has generic fibre which is geometrically integral, then without loss of generality
$f$  has 
generically $\sigma$-geometrically integral fibres. By the proposition,
there is a $\sigma$-localisation $X'$ of $X$ and $Y'$ of $Y$ 
such that $f\restriction X'$ has
$\sigma$-geometrically integral fibres above $Y'$, so in particular $f$ has 
geometrically integral fibres above $Y'$.
\end{proof}

\begin{corollary}\label{locnorm}
Let $(X,\sigma)$ be an $(R,\sigma)$-difference scheme of finite $\sigma$-type
which is separable. There is a $\sigma$-localisation $X'$ of $X$ and $R'$ of $R$
such that $X'/R'$ is normal (in the sense of \ref{defnormal}).
\end{corollary}
\begin{proof}
By \ref{loc-si-smooth}, take a $\sigma$-localisation $X'/R'$ which is $\sigma$-smooth.
\end{proof}

\subsection{Local Properties}\label{local-properties}

This subsection is mostly concerned with the question of whether 
it is reasonable to expect that if a property holds locally, at every point
of a fixed point spectrum of a difference ring, then it also holds globally.
\begin{definition} 
 Let $(M,\sigma)$ be an $(A,\sigma)$-module and let $(N,\sigma)$ be a sumbodule. 
 \begin{enumerate}
 \item We say that $(M,\sigma)$ is \emph{well-mixed} if $am=0$ implies $\sigma(a)m=0$
 for all $a\in A$, $m\in M$.
 \item We say that $(N,\sigma)$ is a \emph{well-mixed submodule} of $(M,\sigma)$
 if the module $M/N$ is well-mixed.
 \end{enumerate}
\end{definition} 

Clearly $(M,\sigma)$ is well-mixed if and only if the annihilator $\mathop{Ann}(m)$ of
any $m\in M$ is a well-mixed $\sigma$-ideal in $(A,\sigma)$. Indeed, if $ab\in \mathop{Ann}(m)$,
then $a(bm)=0$ so $\sigma(a)(bm)=(\sigma(a)b)m=0$ and $\sigma(a)b\in\mathop{Ann}(m)$.

Moreover, since the intersection of well-mixed submodules is well-mixed and $M$ is trivially a well-mixed submodule of itself, for every submodule $(N,\sigma)$ of $(M,\sigma)$ there exists  a smallest well-mixed submodule $[N]_w$ containing $N$. Thus $[0]_w$ is the smallest 
well-mixed submodule of $(M,\sigma)$ associated with
the largest well-mixed quotient $M_w$ of $M$.

\begin{proposition}\label{loczero}
Let $(M,\sigma)$ be a well-mixed $(A,\sigma)$-module.
The following are equivalent.
\begin{enumerate}
\item\label{uno} $M=0$;
\item\label{due} $M_\p=0$ for every $\p\in\spec^\sigma(A)$.
\item\label{tre}  $M_\p=0$ for every $\p$  maximal in $\spec^\sigma(A)$.
\end{enumerate}
\end{proposition}
\begin{proof}
It is clear that (\ref{uno}) implies (\ref{due}) and (\ref{due}) implies (\ref{tre}). 
 Suppose that (\ref{tre})
holds but $M\neq 0$. Let $x\in M\setminus\{0\}$ and let $\A=\mathop{Ann}(x)$. Then
$\A\neq (1)$ is well-mixed and, by \cite[\ref{wmaff}]{ive-tch}, $V^\sigma(\A)\neq\emptyset$. Choose
a maximal $\p$ in $V^\sigma(\A)$. Since $x/1=0$ in $M_\p$, there exists an 
$a\notin\p$ such that $ax=0$, which is in
contradiction with $\mathop{Ann}(x)\subseteq\p$.
\end{proof} 

\begin{corollary}\label{corwkloczero}
Let $(M,\sigma)$ be an $(A,\sigma)$-module. If $M_\p=0$ for every $\p\in\spec^\sigma(A)$, 
then $M_w=0$.
\end{corollary}

The above can be sharpened as follows.

\begin{proposition}\label{wkloczero}
Let $(M,\sigma)$ be an $(A,\sigma)$-module. If $(M_\p)_w=0$ for every $\p$ 
maximal in $\spec^\sigma(A)$, then $M_w=0$.
\end{proposition}
\begin{proof}
Using the universal properties of localisation and passing to well-mixed quotients,
 as well as the fact
that localisation is an exact functor, we construct a commutative diagram
$$
 \begin{tikzpicture} 
\matrix(m)[matrix of math nodes, row sep=2.6em, column sep=2em, text height=1.5ex, text depth=0.25ex]
 { M & M_\p &\\
   M_w &  (M_\p)_w &\\
   & & (M_w)_\p\\}; 
\path[->,font=\scriptsize,>=to, thin]
(m-1-1) edge node[left] {$\pi$} (m-2-1) edge node[above] {$\psi$} (m-1-2)
(m-2-1) edge node[above] {$\alpha$} (m-2-2) edge node[below]{$\psi'$} (m-3-3)
(m-1-2) edge node[right] {$\pi'$} (m-2-2) edge node[right]{$\pi_\p$} (m-3-3)
(m-2-2) edge node[above]{$\beta$} (m-3-3);
\end{tikzpicture}
$$
in which $\pi$ and $\pi'$ are surjective, so we conclude that $\pi_\p$ and 
$\beta$ are also surjective. Therefore, $(M_\p)_w=0$ implies that $(M_w)_\p=0$
and we finish by \ref{loczero}.
\end{proof}

\begin{proposition}\label{locinj}
Let $\phi:(M,\sigma)\to (N,\sigma)$ be an $(A,\sigma)$-module homomorphism and assume
that $(M,\sigma)$ is well-mixed. The following are equivalent.
\begin{enumerate}
\item\label{unoo} $\phi$ is injective;
\item\label{duee} $\phi_\p:M_\p\to N_\p$ is injective for every $\p\in\spec^\sigma(A)$.
\item\label{tree} $\phi_\p:M_\p\to N_\p$ is injective for every $\p$ maximal in $\spec^\sigma(A)$.
\end{enumerate}
\end{proposition}

\begin{proof}
\noindent$(\ref{unoo})\Rightarrow(\ref{duee})$. If $0\to M\to N$ is exact, since localisation
is exact, we get that $0\to M_\p\to N_\p$ is also exact. 
$(\ref{duee})\Rightarrow(\ref{tree})$ is trivial.

\noindent$(\ref{tree})\Rightarrow(\ref{unoo})$. Let $M'=\ker\phi$. Then 
$0\to M'\to M\to N$ is exact so $0\to M'_\p\to M_\p\to N_\p$ is exact for every 
$\p\in\spec^\sigma(A)$. By assumption,  $M'_\p=0$ for every $\p\in\spec^\sigma(A)$.
 Since $M'$ is well-mixed (as a submodule of $M$), by \ref{loczero} we conclude that $M'=0$.
 \end{proof}

\begin{proposition}\label{locwksurj}
Let $\phi:(M,\sigma)\to (N,\sigma)$ be an $(A,\sigma)$-module homomorphism. 
If $\phi_\p:M_\p\to N_\p$ is almost surjective for every $\p$ maximal in $\spec^\sigma(A)$, then
$\phi$ is almost surjective, $[\mathop{\rm im}(\phi)]_w=N$ (equivalently, $\mathop{\rm coker}(\phi)_w=0$).
\end{proposition}

\begin{proof}
Let $N'=\mathop{\rm coker}(\phi)$. Then $M\to N\to N'\to0$ is exact, and by localisation 
$M_\p\to N_\p\to N'_p\to 0$ is exact for every $\p\in\spec^\sigma(A)$. By assumption, 
$(N'_\p)_w=0$ for all $\p$ maximal in $\spec^\sigma(A)$ and \ref{wkloczero} implies that $N'_w=0$.
\end{proof}

\begin{lemma}[\cite{eisenbud-comm-alg}, 6.4]\label{tensorzero}
Let $M$ and $N$ be $A$-modules and suppose $N$ is generated by $\{n_i\}$.
Then every element of $M\otimes_AN$ can be written as $\sum_i m_i\otimes n_i$ with
finitely many nonzero $m_i$
and $\sum_i m_i\otimes n_i=0$ in $M\otimes_AN$ if and only if there exist
$m_j'\in M$ and $a_{ij}\in A$ such that for every $i$,
$$
\sum_j a_{ij}m_j'=m_i%
$$
and for every $j$,
$$
\sum_i a_{ij}n_i=0.%
$$
\end{lemma}

\begin{proposition}\label{tensorwm}
Let $(M,\sigma)$ and $(N,\sigma)$ be $(A,\sigma)$-modules with $(N,\sigma)$ well-mixed.
Then $(M,\sigma)\otimes_{(A,\sigma)}(N,\sigma)$ is well-mixed.
\end{proposition}
\begin{proof}
Pick a set of generators $\{n_i\}$ for $N$. Suppose $b\sum_i m_i\otimes n_i=0$.
Then $\sum_i m_i\otimes bn_i=0$ so \ref{tensorzero} implies the existence of $m_j'\in M$
and $a_{ij}\in A$ such that for every $i$, $\sum_j a_{ij}m_j'=m_i$ and for every $j$,
$0=\sum_i a_{ij}bn_i=b\sum_i a_{ij}n_i$. Since the latter holds in $(N,\sigma)$ which is
well-mixed, we get that $0=\sigma(b)\sum_i a_{ij}n_i=\sum_i a_{ij}\sigma(b)n_i$. Using
\ref{tensorzero} again, it follows that $\sigma(b)\sum_i m_i\otimes n_i=\sum_i m_i\otimes\sigma(b)n_i=0$.
\end{proof}

\begin{proposition}\label{locflat}
Let $(M,\sigma)$ be a well-mixed $(A,\sigma)$-module. The following are equivalent.
\begin{enumerate}
\item\label{unooo} $M$ is a flat $A$-module.
\item\label{dueee} $M_\p$ is a flat $A_\p$-module for every $\p\in\spec^\sigma(A)$.
\item\label{treee} $M_\p$ is a flat $A_\p$-module for every $\p$ maximal in $\spec^\sigma(A)$.
\end{enumerate}
\end{proposition}
\begin{proof}
\noindent$(\ref{unoo})\Rightarrow(\ref{duee})$. Assuming $(i)$, it is classically known that
$M_\p$ is a flat $A_\p$ module for every prime $\p$.
$(\ref{dueee})\Rightarrow(\ref{treee})$ is trivial.

\noindent$(\ref{treee})\Rightarrow(\ref{unooo})$. 
Let $(N,\sigma)\to (P,\sigma)$ be injective. Then $N_\p\to P_\p$ is injective for every 
$\p\in\spec^\sigma(A)$. By assumption, $N_\p\otimes_{A_\p}M_\p\to P_\p\otimes_{A_\p}M_\p$
is injective and thus $\left(N\otimes_AM\right)_\p\to\left(P\otimes_AM\right)_\p$ is injective
for all $\p$ maximal in $\spec^\sigma(A)$. Since $N\otimes_AM$ is well-mixed by \ref{tensorwm},
Proposition~\ref{locinj} implies that $N\otimes_AM\to P\otimes_AM$ is injective.
\end{proof}

\begin{remark}
Let $(A,\sigma)\to(B,\sigma)$ be a homomorphism of well-mixed difference rings such that
$B$ is a flat $A$-module and denote by $\bar{A}$ and $\bar{B}$ the rings of global sections
of $\spec^\sigma(A)$ and $\spec^\sigma(B)$. We can consider $\bar{B}$ as an $A$-module
via the morphism $A\hookrightarrow\bar{A}\to\bar{B}$ as in \ref{embedcataff}, and
we can conclude that $\bar{B}$ is flat over $A$. 
Indeed, by \cite[\ref{wmaff}(\ref{dvaipol})]{ive-tch}, $\bar{B}$ is well-mixed, by  \cite[\ref{wmaff}(\ref{sesttt})]{ive-tch} we know
that $\spec^\sigma(B)\simeq\spec^\sigma(\bar{B})$ and $\bar{B}_{\bar{\p}}\simeq B_\p$,
which suffices to apply \ref{locflat}.
\end{remark}

\begin{proposition}
Let $(A,\sigma)$ be a well-mixed domain. If $A_\p$ is normal for every $\p$
maximal in $\spec^\sigma(A)$,
then $A$ is almost normal.
\end{proposition}

\begin{proof}
Let $K$ be the fraction field of $A$, let $C$ be the integral closure of $A$ in $K$ and
denote by $\phi:A\hookrightarrow C$ the inclusion. By assumption, each $\phi_\p$ is surjective, so 
\ref{locwksurj} implies that $\phi$ is almost surjective and thus $[A]_w=C$.
\end{proof}

\begin{definition}\label{defnormal}
A difference scheme $(X,\Sigma)$ is said to be \emph{normal} if every local
ring $\OO_x$, for $x\in X$, is normal.
\end{definition}

\subsection{\'Etale morphisms of difference schemes}

\begin{definition}\label{defformalsmooth}
A morphism $(R,\sigma)\to(S,\sigma)$ if \emph{formally smooth} (resp.\ 
\emph{formally unramified}, \emph{formally \'etale}), if every solid commutative diagram 
$$
 \begin{tikzpicture} 
\matrix(m)[matrix of math nodes, row sep=2.6em, column sep=2em, text height=1.5ex, text depth=0.25ex]
 { (S,\sigma) & A/I\\
   (R,\sigma) &  (A,\sigma)\\}; 
\path[->,font=\scriptsize,>=to, thin]
(m-2-1) edge  (m-1-1) edge (m-2-2)
(m-2-2) edge  (m-1-2)
(m-1-1) edge (m-1-2) edge[dashed]  (m-2-2);
\end{tikzpicture}
$$
with $I$ a difference ideal with $I^2=0$, there exists
at least one (resp.\ at most one, exactly one) dashed arrow making the
diagram commutative.
\end{definition}
Recall that a morphism of rings $R\to S$ is defined to be formally smooth,
formally unramified or formally \'etale by using exactly the same universal
property in the category of commutative rings, omitting the difference structure.

\begin{lemma}
If $(R,\sigma)\to(S,\sigma)$ is formally smooth, then $R\to S$ is formally smooth.
\end{lemma}
\begin{proof}
Let $(P,\sigma)\to (S,\sigma)$ be a surjective $(R,\sigma)$-algebra morphism from a difference polynomial ring $P$, and let $J$ be the kernel, which is a difference
ideal. Consider the above diagram for $A=P/J^2$ an $I$ generated by $J$.
By formal smoothness, we obtain a (difference) morphism $S\to P/J^2$
which is a right inverse to the surjection $P/J^2\to S$, and thus $R\to S$
is formally smooth using \cite[00TL]{dejong-stacks}. %
\end{proof}

\begin{remark}\label{dercommsi}
For a difference $(R,\sigma)$-algebra $(S,\sigma)$, the module of
relative differentials $\Omega_{S/R}$ naturally classifies $R$-derivations that
commute with $\sigma$.
Indeed, if we let $J$ be the kernel of the multiplication map $S\otimes_RS\to S$,
it is known that $\Omega_{S/R}\simeq J/J^2$. However, in this context
$J$ is a difference ideal and $J/J^2$ comes equipped with a natural
difference structure, which entails in particular that the
universal $R$-derivation $d:S\to\Omega_{S/R}$ satisfies $$d\sigma=\sigma d.$$
\end{remark}

\begin{remark}\label{remgiab}
The above is in stark contrast with the various notions of smoothness 
developed in G.~Giabicani's thesis. With clear intent to apply his theory
to the case where $\sigma$ is a power of the Frobenius automorphism, he postulates
$d\sigma=0$. Another fundamental difference is that \'etale morphisms
in our context as developed below are of relative total dimension $0$, whereas in Giabicani's context
they are of relative transformal dimension $0$. 
\end{remark}

\begin{remark}\label{difffg}
If $(B,\sigma)$ is an $(A,\sigma)$-algebra of finite $\sigma$-type, the second
exact sequence for differentials implies that $\Omega_{B/A}$ is a finitely
$\sigma$-generated $(B,\sigma)$-module.
\end{remark}

\begin{lemma}\label{unramdiff}
Given a difference morphism $(R,\sigma)\to(S,\sigma)$, 
the following statements are equivalent:
\begin{enumerate}
\item\label{un} $(R,\sigma)\to(S,\sigma)$ is formally unramified;
\item\label{do} $R\to S$ is formally unramified;
\item\label{troi} $\Omega_{S/R}=0$.
\end{enumerate}
\end{lemma}
\begin{proof}
In view of \ref{dercommsi}, since $d\sigma=\sigma d$, 
the classical proof of the equivalence of (\ref{do}) and (\ref{troi}) also
works for the equivalence of (\ref{un}) and (\ref{troi}).
\end{proof}

\begin{corollary}
Let $(R,\sigma)\to(S,\sigma)$ be a morphism. The following are equivalent:
\begin{enumerate}
\item $(R,\sigma)\to(S,\sigma)$ is almost formally unramified in the sense that
$(\Omega_{S/R})_w=0$;
\item for every $\q\in\spec^\sigma(S)$ lying over $\p=\q\cap R$, 
$(R_\p,\sigma)\to(S_\q,\sigma)$ is formally unramified;
\item for every $\q\in\spec^\sigma(S)$ lying over $\p=\q\cap R$, 
$R_\p\to S_\q$ is formally unramified.
\end{enumerate}
\end{corollary}
\begin{proof}
Straightforward from \ref{wkloczero} applied to the $(S,\sigma)$-module
$\Omega_{S/R}$.
\end{proof}

\begin{definition}
A morphism $(R,\sigma)\to(S,\sigma)$ is \emph{smooth} (resp.\ \emph{unramified},
\emph{\'etale}), if it is of finite $\sigma$-type and formally smooth (resp.\ 
formally unramified, formally \'etale).
\end{definition}

\begin{proposition}
If $(R,\sigma)\to(S,\sigma)$ is $\sigma$-smooth, then $(R,\sigma)\to(S,\sigma)$
is smooth.
\end{proposition}
\begin{proof}
It suffices to prove that $(R,\sigma)\to(S,\sigma)$ is formally smooth.
Suppose we have a solid part of the diagram from \ref{defformalsmooth}.
Using the functor $[\sigma]_R$, we produce a diagram
$$
 \begin{tikzpicture}
[cross line/.style={preaction={draw=white, -,
line width=4pt}}]
\matrix(m)[matrix of math nodes, row sep=1.1em, column sep=0em, text height=1.3ex, text depth=0.25ex]
{		 &[1em] |(u4)| {[\sigma]_RS}	& [-2em]			& [.5em]|(u3)| {[\sigma]_RA/I}	\\   %
|(u1)|	{(S,\sigma)} &		& |(u2)|{A/I} 	&			\\
		& |(d4)| {}	&			& |(d3)| {[\sigma]_RA}	\\   %
|(d1)|	{(R,\sigma)} &		& |(d2)|{(A,\sigma)} 	&			\\};
\path[->,font=\scriptsize,>=to, thin]
(d1) edge (u1) edge  (d2) 

(u1) edge  (u2)  edge node[auto]{$i_S$} (u4)
	
(d2) edge  (u2)  edge node[above]{$i_A$} (d3)
(u4) edge  (u3)	
(u2) edge[cross line]  (u3)	
(d3) edge  (u3)
(u4) edge[dashed,bend left=10](d3)
(u2) edge[cross line]  (u3)
(d3.south) edge[bend left=20] node[below right=-2pt]{$\pi_A$}(d2)
(u1) edge[dashed] (d2);
\end{tikzpicture}
$$ 
where the dashed arrow $[\sigma]_RS\to[\sigma]_RA$ exists by the assumption
of $\sigma$-smoothness, the morphism $\pi_A$ exists since $A$ already has 
a difference structure, and we can construct a morphism $(S,\sigma)\to(A,\sigma)$
as the obvious composite.
\end{proof}

\begin{definition}
A morphism $f:(X,\Sigma)\to(Y,\sigma)$ is called \emph{smooth} (resp.\ \emph{unramified, \'etale}), if it is of finite transformal type and for every
$x\in X$ and every $\tau\in\Sigma_x$, the morphism 
$(\OO_{f(x)},\sigma)\to(\OO_x,\tau)$ is smooth (resp.\ unramified, \'etale).
\end{definition}

\begin{remark}
A morphism $f:(X,\Sigma)\to(Y,\sigma)$ of finite
transformal type is  unramified (resp.\ \'etale) if and only if
for every $x\in X$ there exists an (affine) open neighbourhood on which $f$ is
modelled by an unramified (resp.\ \'etale) morphism of difference rings.
More precisely, if a morphism $(R,\sigma)\to (S,\sigma)$ is
 unramified (resp.\ \'etale) at some $\q\in\spec^\sigma(S)$ lying over $\p=\q\cap R$,
then there is a $g\notin\q$ such that the $\sigma$-localisation 
$(R,\sigma)\to(S_g,\sigma)$ is unramified (resp. \'etale).

The above `openness' statement for the property of being unramified is obvious 
from \ref{unramdiff} and \ref{difffg}. For \'etaleness, it follows from \ref{difffg} and
the \emph{Jacobian criterion} for smoothness which states the following. 
Let $(A,\sigma)$ be a difference ring, $(B,\sigma)$ a formally smooth 
$(A,\sigma)$-algebra,
$J$ a reflexive difference ideal of $B$ and let $C=B/J$. Then $(C,\sigma)$ is
a formally smooth $(A,\sigma)$-algebra if and only if the natural morphism
of difference modules
$$
\delta:J/J^2\to\Omega_{B/A}\otimes_BC$$
is left-invertible. We omit the details of the proof since these results
will not be used in the sequel.
\end{remark}

In view of Babbitt's decomposition discussed in the next section, 
many considerations reduce to 
a study of unramified or \'etale morphisms with stronger finiteness assumptions
such as quasi-finiteness or finiteness, and we will make every effort to
explicitly state them when possible.

\begin{proposition}\label{trivin-et}
Suppose $(G,\tilde{\Sigma})$ acts admissibly on $(X,\Sigma)$ of finite transformal
type over a difference field $(k,\sigma)$ and
suppose that $(Y,\Sigma_0)=(X,\Sigma)/(G,\tilde{\Sigma})$ and that $X$ is finite
over $Y$.
If $G_i(x)=(e)$ for all $x\in X$, the natural projection $(X,\Sigma)\to(Y,\Sigma_0)$
is finite \'etale.
\end{proposition}
\begin{proof}
Since the assertion is local, we may assume that $X$ is affine and that we are
in the situation of \cite[\ref{quotaff}]{ive-tch}. Suppose $X=\spec^\Sigma(A)$ and
that $(G,\tilde{\Sigma})$ acts on $(A,\Sigma)$. By assumption, there is
a finite tuple $a\in A$ so that $A=k[a]_\Sigma$. Writing $\bar{a}=\{ga:g\in G\}$,
we have that $A=k[\bar{a}]_{{\sigma}}$ for any choice of $\sigma\in\Sigma$. 
Then $G$ acts on each $A_n=k[\bar{a},\sigma\bar{a},\ldots,\sigma^n\bar{a}]$,
and we can form $B_n=A_n^G$. We have 
$A=\varinjlim_nA_n$ and $B=A^G=\varinjlim_nB_n$ and we have formed a projective
limit of Galois covers $p_n:X_n\to Y_n$ such that $X\to Y$ is obtained by
taking the $\Sigma$-fixed points of the ambient Galois covering 
$\varprojlim_nX_n\to\varprojlim_nY_n$. Let $x\in X$ and write $x_n$ for the projection
of $x$ in $X_n$. Since $G_i(x)=(e)$, we have that $G_i(x_n)=(e)$ for all $n$, so
by the known result for Galois covers of locally Noetherian schemes
it follows that $p_n$ is \'etale at $x_n$. By compatibility of formal \'etaleness
with limits, we conclude that $X\to Y$ is \'etale at $x$.
\end{proof}

\begin{corollary}\label{et-trivin}
Suppose $(X,\Sigma)$ is integral and that $(G,\tilde{\Sigma})$ acts faithfully.
Then $X\to X/G$ is \'etale if and only if all the inertia groups of elements of $X$
are trivial.
\end{corollary}
\begin{proof}
By the previous result, it suffices to show that if the quotient morphism is \'etale at 
$x\in X$, then $G_i(x)=(e)$. Take an $x\in X$ with $p:X\to X/G$ \'etale.
We can easily reduce to the case where $G_i(x)=G$ and $\kk(x)=\kk(p(x))$.
Note, since $X$ is integral and $G$ is faithful on $X$,
then $G$ is also faithful on $\OO_x$. However, the classical proof works for the
finite local \'etale  extension $\OO_{p(x)}\to\OO_x$ and shows that $G=(e)$.  
\end{proof}

\subsection{Babbitt's decomposition}

\begin{definition}
A morphism $(S,\sigma)\to(R,\sigma)$ of integral difference rings
is called \emph{benign}
if there exists a quasifinite $S\to R_0$ such that $(R,\sigma)$ is isomorphic
to $[\sigma]_S R_0$ over $(S,\sigma)$. In other words, writing 
$R_{i+1}= R_i\otimes_SS$ for $i\geq0$ (where the morphism $S\to S$
is $\sigma$), $(R,\sigma)$ is the (limit) tensor product of the $R_i$ and
the canonical morphisms $\sigma_i:R_i\to R_{i+1}$ over $S$.

In the \emph{benign Galois} case, $R_0$ is Galois over $S$ with group $G_0$
and the Galois group $\Gal(R/S)=\Gal(\kk(R)/\kk(S))=(G,()^\sigma)$
 is isomorphic to the direct product of 
 $G_i=\Gal(R_i/S)$ and $()^\sigma$ `shifts' from $G_i$ to $G_{i+1}$. 
 \end{definition}
 In view of such a specific form of $G$, for any $h,h' \in G$ there is a $g\in G$
 such that $h'=g^{-1}hg^\sigma$, i.e. $h$ and $h'$ are \emph{$()^\sigma$-conjugate,}
 and we get:
 \begin{lemma}\label{benignliftrng}
 For any $y\in\spec^{\sigma}(S)$, any algebraically closed difference field 
 $(F,\varphi)$ extending $(\kk(y),\sigma^y)$, any $\bar{y}\in (S,\sigma)(F,\varphi)$ above $y$ and any $g\in G$,
 there exists an $\bar{x}\in (R,g\sigma)(F,\varphi)$ lifting $\bar{y}$.
\end{lemma}

\begin{definition}
\begin{enumerate}
\item
A morphism $(\psi,()^\psi):(S,T)\to(R,\Sigma)$ of integral almost-strict difference rings is \emph{benign} if for some (or equivalently, for all) $\sigma\in\Sigma$,
the morphism $(S,\sigma^\psi)\to (R,\sigma)$ is benign.
\item 
A morphism $(X,\Sigma)\to(Y,T)$ of almost-strict difference schemes is \emph{benign} if it is affine and above each open affine subset of $Y$ it is modelled by
a fixed-point spectrum of a benign morphism of rings.
\end{enumerate}
\end{definition}

An immediate consequence of \ref{benignliftrng} is the following.
\begin{lemma}\label{benignlift}
Let $(X,\Sigma)\to(Y,T)$ be an \'etale benign Galois morphism.
For any $\tau\in T$ and $\sigma\in\Sigma$ mapping to $\tau$, any
$y\in Y^\tau$, any algebraically closed difference field $(F,\varphi)$ extending
$(\kk(y),\tau^y)$, any $\bar{y}\in Y^\tau(F,\varphi)$ above $y$, there exists
an $\bar{x}\in X^\sigma(F,\varphi)$ lifting $\bar{y}$. %
\end{lemma}

\begin{definition}
Two difference schemes $(X_1,\Sigma_1)$ and $(X_2,\Sigma_2)$ 
are called \emph{equivalent}, written $(X_1,\Sigma_1)\simeq(X_2,\Sigma_2)$, if they have isomorphic inversive closures, 
$(X_1^{\rm inv},\Sigma_1^{\rm inv})\cong(X_2^{\rm inv},\Sigma_2^{\rm inv})$.
\end{definition}

The following is a slight refinement of a fundamental theorem from
\cite{babbitt}, showing how to use it for not necessarily inversive difference fields.
\begin{lemma}[Babbitt's Theorem]\label{babb-fields}
Let $(K,\sigma)\to (L,\sigma)$ be a $\sigma$-separable Galois extension of finite $\sigma$-type. Then we have a tower
$$
(K,\sigma)\to(L_0,\sigma)\to(L_1,\sigma)\to\cdots\to(L_n,\sigma)\simeq(L,\sigma)
$$
of difference field extensions with $L_0/K$ finite and all $L_{i+1}/L_i$
benign for $i\geq 0$.
\end{lemma}
\begin{proof}
Let us remark that, if $F/K$ is $\sigma$-separable, and 
$F^{\rm inv}=K^{\rm inv}(a)_{\sigma}$
for some $a=a_1,\ldots,a_n$, then there is an $r\geq 0$ such that 
$F=K(\sigma^ra)_\sigma$. 

The original theorem from \cite{babbitt} gives that, writing $\tilde{L}=K^{\rm inv}L$,
and $\tilde{L}_0$ for the core of $K^{\rm inv}$ in $\tilde{L}$, there exist
$u_1,\ldots,u_n\in \tilde{L}$ such that 
$\tilde{L}\simeq \tilde{L}_0(u_1,\dots,u_n)_\sigma$ and for every $0\leq i\leq n-1$,
$\tilde{L}_0(u_1,\ldots,u_{i+1})_\sigma$ is a benign extension of 
$\tilde{L}_0(u_1,\ldots,u_{i})_\sigma$ with normal minimal generator $u_{i+1}$.

It is known that $\tilde{L}_0$ is inversive, so using the above
remark for $F=\tilde{L}_0$, we deduce that $\tilde{L}_0=L_K$, the core of $K$ in $L$.
Bearing in mind that $L\simeq\tilde{L}$, we produce the required decomposition.
\end{proof}

Babbitt's theorem on algebraic extensions of difference fields  has the following consequence in our terminology, providing a deep structure theorem.
\begin{theorem}[Babbitt's decomposition]\label{babb-sch}
Any  generically \'etale quasi-Galois $\sigma$-separable morphism of finite
transformal type $(X,\Sigma)\to (Y,T)$ of normal %
affine almost-strict difference schemes factorizes as 
$$(X,\Sigma)\simeq(X_n,\Sigma_n)\to\cdots\to(X_1,\Sigma_1)\to (X_0,\Sigma_0)\to (Y,T),$$
where $(X_0,\Sigma_0)\to(Y,T)$ is generically finite \'etale quasi-Galois and for 
$i\geq0$, $(X_{i+1},\Sigma_i)\to(X_{i},\Sigma_i)$ is benign Galois. 
Modulo a transformal localisation of $Y$, we can achieve that $(X_0,\Sigma_0)\to (Y,T)$
is finite \'etale quasi-Galois, and that $X_{i+1}\to X_i$ are \'etale benign Galois.
\end{theorem}

\begin{proof}
By applying Babbitt's theorem \ref{babb-fields} 
to the extension of function fields $(\kk(Y),\tau)\to (\kk(X),\sigma)$ for a suitable
choice of $\sigma$ and $\tau$ we obtain a tower of difference field extensions
$$(\kk(Y),\tau)\to (L_0,\sigma)\to(L_1,\sigma)\cdots\to (L_n,\sigma)\simeq(\kk(X),\sigma),$$
where $L_0/\kk(Y)$ is finite and each $(L_{i+1},\sigma)/(L_i,\sigma)$ is benign
for $i\geq0$. We let $\Sigma_i$ be the $\diff$-structure obtained as a restriction of 
$\Sigma$ from $L_n$ to $L_i$. 
Let $(X_i,\Sigma_i)$ be the normalisation of $(Y,T)$ in $(L_i,\Sigma_i)$.
It is clear from \cite[\ref{cor:norm-lattice}]{ive-tch}
 that by a transformal localisation we can
achieve
that $X_0\to Y$ is finite \'etale and it remains to show that each $X_{i+1}\to X_i$
can be made \'etale benign, which is granted by the following lemma.
\end{proof}

\begin{lemma}\label{norm-benign}
Let $(R,\sigma)$ be a normal domain with fraction field $(K,\sigma)$ and let 
$(K,\sigma)\to (L,\sigma)$ be a benign extension of difference fields
such that $L$ is the composite of the linearly disjoint subfields $L_i=\sigma^i(L_0)$
where $L_0$ is a fraction field of an \'etale $R$-algebra $A_0$.
Then the normalisation of $R$ in $L$ is the (limit) tensor product of the $A_0^{\sigma^i}$ over $R$ and thus benign over $R$.
\end{lemma}
\begin{proof}
Since $A_0$ is \'etale over $R$ so is any $A_0^{\sigma^i}$, and any
tensor product of those is therefore $R$-torsion-free and the conclusion
follows from linear disjointness in the spirit of the proof of \ref{preprl}.
\end{proof}

\begin{lemma}
Let $(K,\sigma)\to(L_0,\sigma)\to(L_1,\sigma)$ be a tower of Galois extensions of 
difference fields with $L_0/K$ finite and $L_1/L_0$ benign.
The exact sequence
$$
1\to\Gal(L_1/L_0)\to\Gal(L_1/K)\to\Gal(L_0/K)\to 1
$$
is split.
\end{lemma}
\begin{proof}
Let us start by using shorthand notation 
$H=\Gal(L_1/L_0)$, $G=\Gal(L_1/K)$, $G_0=\Gal(L_0/K)$ and $\pi$
for the projection $G\to G_0$. Let us denote  by $\Phi$ the operator
$()^\sigma:G\to G$. Since $L_1/L_0$ is benign, $H$ has the specific form of 
a direct product of infinitely many  copies of a finite group and $\Phi$ shifts
between the copies. Thus we have that $\cap_i\Phi^iH=(1)$. Writing $\Phi_0$
for the operator induced by $\Phi$ on $G_0$, since $G_0$ is finite, let us
fix some $n$ such that $\Phi_0^n=1$.
The section $s:G_0\to G$ is defined by
$$
s(g_0)=\cap_i(\Phi^n)^i(\pi^{-1}(g_0)).
$$
\end{proof}

\begin{lemma}\label{babsplit}
Let $(K,\sigma)\to(L,\tilde{\sigma})$ be a Galois extension of finite 
$\tilde{\sigma}$-type. Let $L_0=L_K$ be the core of $K$ in $L$.
The exact sequence
$$
1\to\Gal(L/L_0)\to\Gal(L/K)\to\Gal(L_0/K)\to 1
$$
is split.
\end{lemma}
\begin{proof}
Note that the proof of the previous 
lemma works for an arbitrary $L_0/K$ and a finite subgroup $G_0$ of $\Gal(L_0/K)$.
Using Babbitt's decomposition, let
$$
(K,\sigma)\to (L_0,\sigma)\to(L_1,\sigma)\cdots\to (L_n,\sigma)\simeq(L,\sigma),
$$
be a tower of difference field extensions with $L_0/K$ finite and
all $L_{i+1}/L_i$ benign for $i\geq 0$. 
Using the previous lemma we can thus inductively `pull' a copy of $\Gal(L_0/K)$ through the above tower all the way up to $\Gal(L_n/K)$.
\end{proof}

\begin{definition}\label{diffcovfields}
Let $(K,T)\to(F,\Sigma)$ be an extension of difference fields, and let 
$(L,\Sigma)$ be the relative algebraic closure of $K$ in $F$.
We say that an extension $(K,T)\to(F,\Sigma)$ is a \emph{difference covering}
if $L/K$ is Galois and 
$\Sigma$ is a finite set of representatives of the isomorphism classes of
lifts of all $\tau\in T$ to $L$.
\end{definition}

\begin{proposition}\label{domcov}
Let $(K,\sigma)\to(F,\tilde{\sigma})$ be a separable difference field extension of finite
$\tilde{\sigma}$-type. Then it can be subsumed in a difference covering, i.e., there exists a difference field extension $(F,\tilde{\sigma})\to(\tilde{F},\tilde{\sigma})$ and 
an almost strict difference structure $\tilde{\Sigma}\ni\tilde{\sigma}$ on $\tilde{F}$
such that $(K,\sigma)\to(\tilde{F},\tilde{\Sigma})$ is
a difference covering.
\end{proposition}
\begin{proof}
Let $(L,\sigma)$ be the relative algebraic closure of $K$ in $F$. Let $\tilde{L}$
be its normal closure. Let $F'$ be 
the fraction field of
$F\otimes_K\tilde{L}$. %
 Let $L_0\subseteq\tilde{L}_0$, be the cores of $(K,\sigma)$ in $L$ and l$\tilde{L}$,
 respectively.
 We apply \ref{babsplit} to $\tilde{L}/K$ (resp.\ $\tilde{L}/L$) to
find a copy of $G=\Gal(\tilde{L}_0/K)$ (resp.\ $H=\Gal(\tilde{L}_0/L_0)$) in 
$\Gal(\tilde{L}/K)$ (resp.\ $\Gal(\tilde{L}/L)$) and set 
$\tilde{\Sigma}=G\tilde{\sigma}$ (resp.\ $\Sigma=H\tilde{\sigma}$). 
It is known that these represent all isomorphism classes of lifts of $\sigma$ from
$K$ to 
to $\tilde{L}$ (resp.\ from $L$ to $\tilde{L}$). 
We obtain difference coverings
$(L,\tilde{\sigma})\to (\tilde{L},\Sigma)$ and $(K,\sigma)\to (\tilde{L},\tilde{\Sigma})$.
To finish, since $F'$ is regular over $\tilde{L}$, we can lift $\tilde{\Sigma}$
to an extension $\tilde{F}$ of $F'$ by `base field extensions'.  
In more conceptual terms, if
we think of $F'$ as a function field of a geometrically irreducible difference scheme
$X$ over $L$, let $X_0$ be obtained by Weil restriction from $\tilde{L}$ to 
$\tilde{L}^G$. Then the base change of $X_0$ back to $\tilde{L}$ is
isomorphic to $\prod_{g\in G}X^g$, which clearly carries a $\tilde{\Sigma}$-structure
and dominates $X$.
\end{proof}

\begin{definition}
Let $(X,\Sigma)\to (Y,T)$ be a generically dominant morphism of integral
difference schemes. We shall say that it is a \emph{generic difference covering}
if the corresponding inclusion of function fields $(\kk(Y),T)\to(\kk(X),\Sigma)$
is a difference covering of fields.
\end{definition}

About a half of the proof or \ref{domcov} suffices to deduce the following.
\begin{lemma}\label{gal-cl}
If $(X,\Sigma)\to(Y,T)$ is separable algebraic, then it is dominated by
the quasi-Galois closure $(\tilde{X},\tilde{\Sigma})\to (Y,T)$ of $X$ over $Y$,
i.e., the $(Y,T)$-morphism $(\tilde{X},\tilde{\Sigma})\to (X,\Sigma)$ is
a generic difference covering.
\end{lemma}

\begin{proposition}\label{subsgencov}
Let $(X,\Sigma)\to (Y,T)$ be a generically dominant generically smooth
 morphism of integral normal
almost strict difference schemes. Then it can be subsumed in a generic difference
covering in the sense of a diagram
$$
\begin{tikzpicture} 
\matrix(m)[matrix of math nodes, row sep=1.5em, column sep=.5em] %
 {                       & |(P)|{(\tilde{X},\tilde{\Sigma})} & [1em] |(2)| (Z,\Sigma_Z)         \\
 |(1)|{(X,\Sigma)} &                         &            \\
                & |(h)|{(Y,T)}            &\\}; 
\path[->,font=\scriptsize,>=to, thin]
(P) edge  (1) edge (2) edge (h)
(1) edge   (h)
(2) edge  (h);
\end{tikzpicture}
$$
where $(\tilde{X},\tilde{\Sigma})\to (Y,T)$ is  the quasi-Galois closure of 
$(X,\Sigma)\to (Y,T)$, and 
$(Z,\Sigma_Z)\to (Y,T)$ is a generic difference covering.
\end{proposition}
\begin{proof}
In order to simplify notation, let us treat the case of strict difference
scheme morphism $(X,\sigma)\to(Y,\sigma)$. 
Let $L$ be the relative algebraic closure
of $\kk(Y)$ inside $\kk(X)$. As in the proof of \ref{domcov}, we may reduce to
the case where $L/\kk(Y)$ is Galois, and we find a difference structure $\Sigma$
on $L$ which represents all lifts of $\sigma$ to the core of $\kk(Y)$ in $L$.
Let $(\tilde{Y},\Sigma)$ be the normalisation of $(Y,\sigma)$ in $(L,\sigma)$.
Our task is to lift the structure $\Sigma$ to a difference scheme as closely related
to $(X,\sigma)$ as possible.
Let $\iota:\{\sigma\}\to\Sigma$. We construct a $\Sigma$-difference scheme
$\iota_!(X/\tilde{Y})$ by imitating the definition \ref{tensindinj} except
that the underlying space is the fibre product of copies of $X$ over $\tilde{Y}$.
Using the universal properties of functors $\iota^*$ and $\iota_!$,
we obtain a diagram
$$
 \begin{tikzpicture} 
 [cross line/.style={preaction={draw=white, -,
line width=3pt}}]
\matrix(m)[matrix of math nodes, %
row sep=1em, column sep=1em, text height=1.6ex, text depth=0.25ex]
 { 
				& |(5m)|{\iota_!(X/\tilde{Y})}	&			\\[.2em]
|(4l)|{X}			&						&|(4r)|{\iota_!X}\\[.5em]
				& |(3m)|{\tilde{Y}}			&			\\[.2em]
|(2l)|{\iota^*\tilde{Y}}	&					&|(2r)|{\iota_!\iota^*\tilde{Y}}\\[1.2em]
|(1l)|{Y}			&						&|(1r)|{\iota_!Y}\\
};
\path[->,font=\scriptsize,>=to, thin]
(5m) edge (3m) edge (4r)
(4l) edge (5m) edge[cross line] (4r) edge (2l)
(3m) edge[bend left=25] (1l) edge (2r)
(2l) edge (3m) edge[cross line] (2r) edge (1l)
(4r) edge (2r)
(2r) edge (1r)
(1l) edge (1r)
;
\end{tikzpicture}
$$
where the composite $({\iota_!(X/\tilde{Y})},\Sigma)\to(\tilde{Y},\Sigma)\to (Y,\sigma)$
is a generic difference covering.
\end{proof}

\section{Bi-fibered structure of the category of difference schemes}\label{sect:bi-fib}

\begin{definition}[Pullback]
Let $(Y,T)$ be a difference scheme and let $\psi:\Sigma\to T$ be a $\diff$-morphism.
The \emph{pullback} of $Y$ with respect to $\psi$ is defined as 
\[
\psi^*Y=\cup_{\sigma\in\Sigma}Y^{\psi(\sigma)},
\]
with its induced structure as a $\Sigma$-difference scheme, with $\sigma\in\Sigma$
acting as $\psi(\sigma)$ on $Y$. There is a natural morphism
$$
\psi^*Y\to Y.
$$
\end{definition}

The following definitions make sense for full difference structures $\Sigma$, recall
\cite[\ref{def-regular}]{ive-tch}.

\begin{definition}\label{indinj}
Let $(X_0,\Sigma_0)$ be a difference scheme and let 
$\iota:\Sigma_0\hookrightarrow\Sigma$ where $\Sigma$ has generalised conjugation
so that for each $\sigma,\tau\in\Sigma$, there exists a $\tau'$ such that
$\lrexp{\tau}{\sigma}{{\tau'}}\in\Sigma_0$. Let $\tau_i$, $i\in I$ be the
representatives of $\Sigma/\Sigma_0$, i.e., for each $\sigma\in\Sigma$, and
each $\tau_i$, there is a unique $\tau_j$ with $\lrexp{{\tau_j}}{\sigma}{{\tau_i}}\in\Sigma_0$. The assignment $i\mapsto j$ defines a permutation $\bar{\sigma}$ of $I$.
Consider the space
$$
\coprod_{i\in I} X_i,
$$
where each $X_i$ is a copy of $X_0$ (that should be thought of as $\tau_iX_0$)
and $\sigma\in\Sigma$ takes $X_i$ to $X_{\bar{\sigma}(i)}$, and acts as 
$\lrexp{{\tau_{\bar{\sigma}(i)}}}{\sigma}{{\tau_i}}$ on the associated copy of $X$. 
The underlying space of the (coproduct) pushforward is
$$
\iota_*X_0=\bigcup_{\sigma\in\Sigma}\Bigl(\coprod_{i\in I} X_i\Bigr)^\sigma=
\coprod_i\bigcup_{\sigma^{\tau_i}\in\Sigma_0}X_0^{\sigma^{\tau_i}},
$$
and the action of $\Sigma$ is inherited from $\coprod_{i\in I} X_i$.

There is a natural (inclusion) morphism $$(X_0,\Sigma_0)\to (\iota_*X_0,\Sigma).$$
\end{definition}
Note that $\iota_*X_0$ can be set-wise defined by using just the basic 
$\diff$-structure on $\Sigma$, but the action of $\Sigma$ requires 
generalised conjugation.

\begin{definition}\label{tensindinj}
Let $(X_0,\Sigma_0)$ be a difference scheme and let 
$\iota:\Sigma_0\hookrightarrow\Sigma$ and $I$ satisfy the assumptions of
\ref{indinj}.
For $\sigma\in\Sigma$, let $\bar{\sigma}$ denote the permutation of $I$
defined by the requirement that $\bar{\sigma}(i)$ is the unique element of $I$
satisfying
$$
\lrexp{{\tau_{\bar{\sigma}(i)}}}{\sigma}{{\tau_i}}\in\Sigma_0.
$$
Consider the space
$$
\prod_{i\in I} X_i,
$$
where each $X_i$ is a copy of $X_0$, which can be identified  with $X_0^I$, the
space of functions $f:I\to\coprod_{i\in I}X_i$ with the property $f(i)\in X_i$.
We define the action in a natural way with respect to loc.~cit., 
$$
(\sigma f)(i)=\lrexp{{\tau_{\bar{\sigma}(i)}}}{\sigma}{{\tau_i}}(f(\bar{\sigma}(i))).
$$
The underlying space of the (product) pushforward is
$$
\iota_!X_0=\bigcup_{\sigma\in\Sigma}\Bigl(\prod_{i\in I} X_i\Bigr)^\sigma,
$$
and the action of $\Sigma$ is inherited from $\prod_{i\in I} X_i$.

There is a natural (diagonal) morphism $$(X_0,\Sigma_0)\to (\iota_!X_0,\Sigma).$$
\end{definition}

\begin{definition}\label{indsurj}
Let $(X,\Sigma)$ be a difference scheme and let $\pi:\Sigma\to T$ be
a $\diff$-morphism such that there exists a finite group $K$ acting faithfully
on $\Sigma$ such that $\pi$ can be identified with the canonical projection
$\Sigma\to\Sigma/K=T$. 
We define 
$$
\pi_*(X,\Sigma)=(X,\Sigma)/K,
$$
considered as a $T$-difference scheme.
There is an obvious quotient morphism
$$
(X,\Sigma)\to (\pi_*X,T).
$$
\end{definition}

\begin{remark}
Let $\pi:\Sigma\to T$ be as in \ref{indsurj}. We shall not need the functor
$\pi_!$ in the sequel, but we give an idea of its construction on an affine model.
Let $(R,\Sigma)$ be a difference algebra over a difference field $(k,\varphi)$. 
We let 
$$
\pi_!R=R_K=k\otimes_{k[K]}R,
$$
the ring of $K$-coinvariants, with its natural $T$-action. Both natural morphisms
$R\to k\otimes_{k[K]}R$ and $k\otimes_{k[K]}R\to R$ are used to prove the
required adjunction below.
\end{remark}

\begin{definition}\label{induction}
Let $(X,\Sigma)$ be a difference scheme and let $\psi:\Sigma\to T$ be
a $\diff$-morphism which is a composite of $\diff$-morphisms satisfying
the requirements of \ref{indinj} or \ref{indsurj}, $\psi=\psi_1\circ\cdots\psi_n$.
We define
$$
\psi_*X=\psi_{1*}\cdots\psi_{n*}X\ \ \text{ and }\ \ \psi_!X=\psi_{1!}\cdots\psi_{n!}X
$$
and there are obvious natural  morphisms
$$
X\to\psi_*X\ \ \text{ and }\ \ X\to\psi_!X.
$$
\end{definition}

\begin{theorem}
Let $\psi:\Sigma\to T$ be a $\diff$-morphism as in \ref{induction}.
\begin{enumerate}
\item\label{indjen}
The functor $\psi_*$ is left adjoint to $\psi^*$, i.e.,
$$
\Hom_{\Sigma}(X,\psi^*Y)\simeq \Hom_{T}(\psi_*X,Y),
$$
functorially in $(X,\Sigma)$ and $(Y,T)$.
\item\label{indva}
The functor $\psi_!$ is right adjoint to $\psi^*$, i.e.,
$$
\Hom_{\Sigma}(\psi^*Y,X)\simeq \Hom_{T}(Y,\psi_!X),
$$
functorially in $(X,\Sigma)$ and $(Y,T)$.
\end{enumerate}
\end{theorem}
\begin{proof}
\noindent{\bf (\ref{indjen})} 
Suppose we have a difference scheme $(X,\Sigma)$ and let
$\iota:\Sigma_0\hookrightarrow\Sigma$ satisfying the assumptions of \ref{indinj}.
Writing $(X_0,\Sigma_0)=\iota^*X$ and adopting the notation from loc.~cit., 
there is a natural morphism $\iota_*\iota^*X\to X$, induced by the morphism
$\coprod_iX_i\to X$, taking the $i$-th copy
of $X_0$ to $\tau_i(X_0)$. Since $\iota^*\iota_*Y$ is a disjoint union of copies
of $(Y,\Sigma_0)$, an obvious morphism $(Y,\Sigma_0)\to \iota^*\iota_*Y$ is
the inclusion onto the first copy. This makes it easy to verify that the resulting adjunctions
$$
\iota_*\iota^*\to 1\ \ \text{ and }\ \ 1\to\iota^*\iota_*
$$
indeed satisfy the required unit-counit identities for the required adjunction to hold.
When $\pi:\Sigma\to T$ satisfies the requirements of \ref{indsurj}, 
$\pi_*\pi^*\to 1$ is an isomorphism and $1\to\pi^*\pi_*$  is essentially the 
quotient morphism, so the unit-counit relations are easily verified.

\noindent{\bf (\ref{indva})} With the above notation, a natural morphism
$X\to\iota_!\iota^*X$ is obtained as a restriction of the twisted diagonal embedding
$X\to\prod_i X_i$, where the $X_i$ are copies of $X$ and the morphism is
$x\mapsto (\tau_i(x))$. For $(Y,\Sigma_0)$, a natural morphism
$\iota^*\iota_!Y\to Y$ is the projection on the first factor, $\iota^*\iota_!Y$ being
a direct product of copies of $(Y,\Sigma)$.
It is just a formality to verify that the resulting adjunctions
$$
1\to\iota_!\iota^*\ \ \text{ and }\ \ \iota^*\iota_!\to 1
$$
are as required.
\end{proof}

\begin{remark}
As hinted in \cite[\ref{fibrediff}]{ive-tch}, the functor $(X,\Sigma)\mapsto\Sigma$
makes the category of difference schemes into a (split) fibered category
over $\diff$. 

In Grothendieck's terminology from \cite{sga1}, the existence of left adjoints for the pullback functors
makes the category of almost strict difference schemes into a (split)
\emph{bi-fibered} category over the category of almost strict difference structures.
The author is uncertain on the nomenclature of (split) fibrations in which
the pullback functors come with right adjoints as well.
\end{remark}

\section{Effective difference algebraic geometry}\label{s:effective}

As mentioned in the Introduction, one of the main benefits of our Galois
stratification procedure is that it makes the \emph{quantifier elimination} and
\emph{decision} procedures for fields with Frobenii \emph{effective} in an adequate sense
of the word to be expounded in this section.

Ideally, we would like to prove that it makes those procedures
\emph{primitive recursive}, which would represent a significant improvement
on the known results \cite{angus}, \cite{zoe-udi}, \cite{udi}, where it was
shown that the decision procedure is recursive.

Unfortunately, due to the underdeveloped state of constructive difference
commutative algebra, and the lack of algorithms for relevant operations
with difference polynomial ideals, all we can do at the moment is to
show that our Galois stratification procedure is primitive recursive reducible
to a number of basic operations in difference algebra, which we strongly 
believe to be primitive recursive themselves.
 
A ring $(R,\sigma)$ is said to be 
 \emph{effectively presented,}
if it has a finite $\sigma$-presentation over $\Z$, with its generators and relations
explicitly given.
The following is a list of elementary operations on effectively presented rings
that we shall have recourse to
in the sequel. 
\begin{enumerate}[label=$(\dag_{\arabic*})$]
\item\label{E1} Given a difference ideal $I$ in a difference polynomial ring over an
effectively presented difference field, find its minimal
associated $\sigma$-primes, i.e., find an irredundant decomposition 
$\{I\}_\sigma=\p_1\cap\cdots\cap\p_n$  (as in \cite[\ref{strictcomp}]{ive-tch}).
\item\label{E2} Given an extension $(K,\sigma)\to (L,\sigma)$ of effectively presented
 difference fields of finite $\sigma$-type, compute the relative algebraic closure of $K$ in $L$.
\item\label{E3}Given an $\sigma$-separable Galois extension $(K,\sigma)\to (L,\sigma)$ of effectively presented difference
fields of finite $\sigma$-type, compute its Babbitt's decomposition (as in \ref{babb-fields}). 
\item\label{E4} For an effectively presented integrally closed domain $(R,\sigma)$ with fraction field 
$(K,\sigma)$, and an extension $(L,\sigma)$ of $(K,\sigma)$ of finite $\sigma$-type,
find the integral closure $(S,\sigma)$ of $R$ in $L$, and compute the 
$\sigma$-localisation
$(R',\sigma)$ of $R$ so that the corresponding $S'$ is of finite $\sigma$-type over $R'$ (cf.~\ref{cor:norm-lattice}).
\item\label{E5}Given an effectively presented morphism $f:(R,\sigma)\to(S,\sigma)$ 
of effectively presented difference
rings and a suitable property $P$ of scheme morphisms, if
$f$ is generically $\sigma$-$P$,
compute  the $\sigma$-localisations $R'$ of $R$ and $S'$ of $S$ such that
$(R',\sigma)\to(S',\sigma)$ is $\sigma$-$P$ (in particular, we need effective versions
of 
\ref{loc-si-smooth}, \ref{loc-si-integralfibres} and \ref{locnorm}).
\item\label{E6} Given an algebraic extension $(K,\sigma)\to (L,\sigma)$ of 
effectively presented difference
fields of finite $\sigma$-type, compute the quasi-Galois closure of $L$ over $K$.
\item\label{E7} For a finite Galois extension $(K,\sigma)\to (L,\Sigma)$ of 
effectively presented difference
fields, establish an effective correspondence between the
intermediate field extensions and subgroups of the Galois group.
\item\label{E8} Effective Twisted Lang-Weil estimate. In the situation of
\cite[\ref{udiLW}]{ive-tch}, compute explicitly the constant $C$ and the localisation
$S'$ of $S$.
\end{enumerate}
 
 \begin{definition}
 We define \emph{$\dag$-primitive recursive functions} as functions primitive recursive reducible to basic operations in Difference Algebraic Geometry
 as detailed by the following axioms. 

\noindent Basic $\dag$-primitive recursive functions are:
\begin{enumerate}
\item \emph{Constant} functions, \emph{Successor} function $S$, coordinate \emph{Projections};
\item Elementary operations in \emph{difference algebraic geometry}
\ref{E1}--\ref{E8}.
\end{enumerate}
More complex $\dag$-primitive recursive functions are built using:
\begin{enumerate}[resume]

\item \emph{Composition.} If $f$ is an $n$-ary $\dag$-primitive recursive function, and 
$g_1,\ldots,g_n$ are $m$-ary $\dag$-primitive recursive function, then
$$
h(x_1,\ldots,x_m)=f(g_1(x_1,\ldots,x_m),\ldots,g_n(x_1,\ldots,x_m))
$$
is $\dag$-primitive recursive.
\item \emph{Primitive recursion.} Suppose $f$ is an $n$-ary and $g$ is an $(n+2)$-ary 
$\dag$-primitive recursive function. The function $h$, defined by
\begin{align*}
&h(0,x_1,\ldots,x_n) = f(x_1,\ldots,x_n)\\
&h(S(y),x_1,\ldots,x_n)=g(y,h(y,x_1,\ldots,x_n),x_1,\ldots,x_n)
\end{align*}
is $\dag$-primitive recursive.
\end{enumerate}
\end{definition}

\begin{remark}
The operations \ref{E5}--\ref{E8}  
are primitive recursive.
\end{remark} 
\begin{proof}
In view of the constructive nature of proofs of \ref{preprl}, \ref{sigma-generic},
\ref{loc-si-smooth}, \ref{loc-si-integralfibres}, \ref{locnorm}, the fact that
the operation \ref{E5} %
is primitive
recursive will follow from the existence of the 
classical primitive recursive procedure for
finding a localisation satisfying the property $P$ at the start of the prolongation
sequence. 

The operation \ref{E6} is  primitive recursive
because the construction of quasi-Galois closure is
primitive recursive in the algebraic case.  Indeed, if $L=K(a_1,\ldots,a_n)_\sigma$,
and $K(b_1,\ldots,b_m)$ is the quasi-Galois closure of $K(a_1,\ldots,a_n)$,
then $K(b_1,\ldots,b_m)_\sigma$ is the quasi-Galois closure of $L$.

Since the operation \ref{E7} only deals with finite Galois extensions, it follows
that it is primitive recursive by the discussion in \cite{fried-sacer}.

Regarding \ref{E8}, Hrushovski indicates in the `Decidability' subsection following \cite[13.2]{udi} that it is effective, showing how to explicitly compute the constants
for the error term.
 
\end{proof}

\begin{conjecture}
All the operations \ref{E1}--\ref{E4} are primitive recursive. The notions
of \emph{primitive recursive} and \emph{$\dag$-primitive recursive} coincide.
\end{conjecture}
It is plausible to the author that the relevant operations in difference
algebra will be shown to be primitive recursive in the near future. Consequently, 
our Galois stratification, as well as the decision procedure for fields with Frobenii will be shown to be primitive recursive.
 
\section{Galois stratification}
\subsection{Galois stratifications and Galois formulae}

\begin{definition}
Let $(X,\sigma)$ be an $(S,\sigma)$-difference scheme. It is often useful to
consider its \emph{realisation functor} $\tilde{X}$. For each $s\in S$ and
each algebraically closed difference field $(F,\varphi)$ extending $(\kk(s),\sigma^s)$,
$$
\tilde{X}(s,(F,\varphi))=X_s(F,\varphi).
$$
An \emph{$(S,\sigma)$-subassignment} of $X$ is any subfunctor $\cF$ of $\tilde{X}$.
Namely, for any $(s,(F,\varphi))$ as above,
$$
\cF(s,(F,\varphi))\subseteq X_s(F,\varphi),
$$
and for any $u:(s,(F,\varphi))\to(s',(F',\varphi'))$, 
$\cF(u)$ is the restriction of $\tilde{X}(u)$ to $\cF(s,(F,\varphi))$.
\end{definition}

\begin{definition}Let $(S,\sigma)$ be a difference scheme %
and let 
$(X,\sigma)$ be a difference scheme over $(S,\sigma)$.
A \emph{normal (twisted) Galois stratification} 
$$ 
\cA=\langle X, Z_i/X_i, C_i\, |\, i\in I\rangle
$$
of $(X,\sigma)$ over $(S,\sigma)$ is a partition of $(X,\sigma)$
into a finite set of of integral normal  $\sigma$-locally closed difference $(S,\sigma)$-subvarieties 
$(X_i,\sigma)$ of $(X,\sigma)$,
each equipped with a connected Galois covering $(Z_i,\Sigma_i)/(X_i,\sigma)$ with group 
$(G_i,\tilde{\Sigma}_i)$, and   
$C_i$ is a `conjugacy domain' in $\Sigma_i$, as in 
\cite[Section~\ref{sect:galois}]{ive-tch}.
\end{definition}
A normal Galois stratification is \emph{effectively given}, if the base $(S,\sigma)$ and
all the pieces $Z_i$, $X_i$ are affine with  effectively presented coordinate rings
(i.e., of finite $\sigma$-presentation over $\Z$).

\begin{definition}
We define the \emph{(twisted) Galois formula} over $(S,\sigma)$ associated with the above stratification 
$\cA$ to be its realisation subassignment $\tilde{\cA}$ of $X$. 
Given a point $s\in S$ and
an algebraically closed
 difference field $(F,\varphi)$ extending $(\kk(s),\sigma^s)$, 
$$
\tilde{\cA}(s,(F,\varphi))=\cA_s(F,\varphi)=\bigcup_i\{x\in X_{i,s}(F,\varphi)\,\, |\,\, \varphi_{x}^{Z_i/X_i}\subseteq C_i\},
$$
where $\varphi_{x}^{Z_i/X_i}$ denotes the local $\varphi$-substitution at $x$,
as defined in \cite[Subsection~\ref{ss:decompinert}]{ive-tch}.

It can be beneficial to think of the Galois formula associated with $\cA$ and a given \emph{parameter} $s\in S$ as of the formal expression
$$
\theta(t;s)\equiv\{t\in X_s\ |\ \ar(t)\subseteq\con(\cA)\},
$$
whose \emph{interpretation} in any given $(F,\varphi)$ extending $(\kk(s),\sigma^s)$
is given above. Namely, the `Artin symbol' $\ar(t)$ of a point $x\in X_{i,s}(F,\varphi)$
is interpreted as $\ar(x)=\varphi_x$, the local $\varphi$-substitution at $x$ with respect to the covering $Z_i/X_i$, and $\con(\cA)$ at $x$ becomes the appropriate $C_i$. 
\end{definition}

\begin{remark}
If we fix a lift $\sigma_i\in\Sigma_i$ of $\sigma$ for each $i$, the
above data is equivalent to fixing for each $i$ 
 a $()^{\sigma_i}$-conjugacy domain $\dot{C}_i$ in $G_i$, i.e., a union of $()^{\sigma_i}$-conjugacy classes in  $G_i$. This justifies the adjective `twisted' used alongside `stratification'. Clearly,
 $$
 \cA_s(F,\varphi)=\bigcup_i\{x\in X_{i,s}(F,\varphi)\,\, |\,\, \dot{\varphi}_{x}^{Z_i/X_i}\in \dot{C}_i\}.
$$ 
\end{remark}

\begin{remark}
In view of our previous consideration of constructible functions on $(X,\sigma)$, it 
 is clear that a constructible function has only finitely many values
 and that Galois formulae arise as `level-sets' of constructible functions.
 In fact, by identifying a conjugacy domain $C$ with its characteristic function $1_C$,
 we can think of a Galois formula associated with 
 $$ 
\cA=\langle X, Z_i/X_i, C_i\, |\, i\in I\rangle
$$
as the preimage $1_{\cA}^{-1}(1)$ of the constructible function
$$ 
1_{\cA}=\langle X, Z_i/X_i, 1_{C_i}\, |\, i\in I\rangle.
$$
Alternatively, $1_{\cA}$ can be thought of as a characteristic function of $\cA$ on $X$.
With this duality in mind, starting from the \emph{Boolean ring} of characteristic
functions of Galois formulae on $X$ (which is the subring of idempotents in the algebra $\cC(X)$ of
all constructible functions, see \cite[\ref{constr-algebra})]{ive-tch}, we can define a \emph{Boolean algebra}
structure on the class (really a set) of Galois formulae on a given difference
scheme $(X,\sigma)$ over $(S,\sigma)$:
\begin{enumerate}
\item $0_X=\langle X,X/X,0\rangle$, $1_X=\langle X,X/X,1\rangle$;
\item $1_{\cA\land\cB}=1_{\cA}\cdot 1_{\cB}$;
\item $1_{\cA\lor\cB}=1_{\cA}+1_{\cB}-1_{\cA}\cdot1_{\cB}$;
\item $1_{\lnot\cA}=1_X-1_{\cA}$.
\end{enumerate}
\end{remark} 

Although efficient, the above definition of the Boolean algebra structure on
Galois formulae on a given difference scheme can be made more explicit and
informative as follows.

\begin{definition}
Let $f:(X,\sigma)\to(Y,\sigma)$ be an $(S,\varphi)$-morphism, let
$\cA=\langle X,Z_i/X_i,C_i\rangle$ be an $(S,\sigma)$-Galois stratification
on $X$ and let
$\cB=\langle Y,W_j/Y_j,D_j\rangle$ be an $(S,\sigma)$-Galois stratification
on $Y$. 
\begin{enumerate}
\item With notation of \cite[\ref{def:inflat}]{ive-tch}, 
the \emph{inflation} of $\cA$ is defined
as 
$$\cA'=\langle X,Z'_i/X_i,\pi_i^{-1}(C_i)\rangle,$$
and has the property that for every $s\in S$, and every algebraically closed
$(F,\varphi)$ extending $(\kk(s),\sigma^s)$,
$$
\cA'_s(F,\varphi)=\cA_s(F,\varphi).
$$
\item With notation of \cite[\ref{def:refin}]{ive-tch}, the \emph{refinement} of $\cA$ is defined
as
$$\cA'=\langle X,Z_{ij}/X_{ij},\iota_{ij}^{-1}(C_i)\rangle,$$
and has the property that for every $s\in S$, and every algebraically closed
$(F,\varphi)$ extending $(\kk(s),\sigma^s)$,
$$
\cA'_s(F,\varphi)=\cA_s(F,\varphi).
$$
\item With notation of \cite[\ref{def:pullbk}]{ive-tch}, the \emph{pullback} $f^*\cB$ of $\cB$ with respect to
$f$ is defined
as a refinement of
$$\langle X,Z_{j}/X_{j},\iota_{j}^{-1}(D_j)\rangle$$
to a normal refinement of the stratification $X_j$ of $X$. 
It has the property that for every $s\in S$, and every algebraically closed
$(F,\varphi)$ extending $(\kk(s),\sigma^s)$,
$$
f^*\cB_s(F,\varphi)=f_s^{-1}(\cB_s(F,\varphi)).
$$
\end{enumerate}
\end{definition}

\begin{definition} 
Let $(X,\sigma)$ be an $(S,\sigma)$-difference scheme.
The class of  $(S,\sigma)$-Galois formulae on $X$ has a \emph{Boolean algebra} 
structure as follows. 
\begin{enumerate}
\item $\bot_X=\langle X,X/X,\emptyset\rangle$, $\top_X=\langle X,X/X,\{\sigma\}\/\rangle$.
\end{enumerate}
For Galois formulae on $X$ given by $\cA$ and $\cB$, upon a refinement
and an inflation we may assume that $\cA=\langle X,Z_i/X_i,C_i\rangle$ and
$\cB=\langle X,Z_i/X_i,D_i\rangle$, with $C_i,D_i\subseteq\Sigma_i$.
\begin{enumerate}[resume]
\item $\cA\land\cB=\langle X,Z_i/X_i,C_i\cap D_i\rangle$.
\item $\cA\lor\cB=\langle X,Z_i/X_i,C_i\cup D_i\rangle$.
\item $\lnot\cA=\langle X,Z_i/X_i,\Sigma_i\setminus C_i\rangle$.
\end{enumerate}
\end{definition}

\subsection{Direct Image Theorems}

The following result can be considered as a difference version of Chevalley's
theorem stating that a direct image of a constructible set by a scheme morphism
of finite presentation is again constructible.

\begin{proposition}\label{chevalley}
Let $f:(X,\Sigma)\to (Y,T)$ be a generic difference covering of finite transformal type.
Then $f(X)$ contains a dense open subset of $Y$. If $f$ is effectively presented,
we can compute it in a $\dag$-primitive recursive way.
\end{proposition}
\begin{proof}
It is enough to consider the case when $T=\{\sigma\}$.
By $\sigma$-localising (\ref{loc-si-smooth}), 
we may assume that $f$ is a $\sigma$-smooth morphism
between normal difference schemes. This is $\dag$-primitive recursive by \ref{E8}.
By considering the normalisation $\tilde{Y}$ in the relative algebraic closure of $\kk(Y)$ inside
$\kk(X)$, we obtain a baby Stein factorisation $(X,\Sigma)\to (\tilde{Y},\Sigma)\to 
(Y,\sigma)$, 
where the first map has generically geometrically integral fibres, and the second is 
generically $\sigma$-\'etale with $\kk(\tilde{Y})/{\kk(Y)}$ Galois. By localising futher
using \ref{loc-si-integralfibres}, we may assume
that all the fibres of the first morphism are non-empty so the first morphism
is surjective, and we can restrict our attention to the second morphism. These steps
are $\dag$-primitive recursive by \ref{E2}, \ref{E4}, \ref{E5}.

Using \ref{loc-si-smooth}, by another localisation we restrict to the case 
where $(\tilde{Y},\Sigma)\to(Y,\sigma)$ is $\sigma$-\'etale.
Applying Babbitt's decomposition \ref{babb-sch} to $(\tilde{Y},\Sigma)\to(Y,\sigma)$ and a further localisation if necessary, 
we obtain a tower 
$$
\tilde{Y}\simeq Y_n\to\cdots\to Y_1\to Y_0\to Y,
$$
with $(Y_0,\Sigma)\to (Y,\sigma)$ finite Galois and 
all $Y_{i+1}\to Y_i$ benign, for $i\geq 0$. In the effective case, all this can be
achieved in a $\dag$-primitive recursive way, using \ref{E5}, \ref{E3}, \ref{E4}.
The first morphism is a Galois covering and therefore surjective by 
\cite[\ref{quotaff}]{ive-tch}
and its generalisations,
and benign morphisms are clearly surjective. Thus, we conclude that $f$
can be made surjective upon a finite $\sigma$-localisation, which is
enough to deduce the required statement.
\end{proof}

\begin{definition}
Let $(S,\sigma)$ be a normal integral difference scheme of finite $\sigma$-type
over $\Z$, and let $(X,\sigma)$ be an $(S,\sigma)$-difference scheme.
Let $\cF$ and $\cF'$ be $(S,\sigma)$-subassignments of $X$.
We shall say that $\cF$ and $\cF'$ are \emph{equivalent with respect to fields with
Frobenii} over $S$ and write
$$
\cF\equiv_S\cF',
$$
if for every closed $s\in S$, every sufficiently large
finite field $k$ with $(\bar{k},\varphi_k)$ extending $(\kk(s),\sigma^s)$, 
$$
\cF(s,(\bar{k},\varphi_k))=\cF'(s,(\bar{k},\varphi_k)).
$$
\end{definition}

\begin{definition}
Let $f:(X,\sigma)\to (Y,\sigma)$ be a morphism of $(S,\sigma)$-difference
schemes and let $\cA$ be Galois stratification on $X$, associated with a Galois
formula $\chi(x;s)\equiv\{x\in X_s\ |\ \ar(x)\subseteq\con(\cA)\}$. For $s\in S$ and
$(F,\varphi)$ an algebraically closed difference field extending $(\kk(s),\sigma^s)$,
we define a subassignment $f_{\exists}\cA$ of $Y$ by the rule
$$
f_{\exists}\cA(s,(F,\varphi))=(f_{\exists}\cA)_s(F,\varphi)=f_s(\cA_s(F,\varphi))\subseteq Y_s(F,\varphi).
$$
It can also be considered
as an expression 
$$\upsilon(y;s)\equiv\{y\in Y_s\ |\ \exists x\ \chi(x;s), f_s(x)=y\}$$
which justifies the notation somewhat.
\end{definition}

\begin{lemma}\label{galtwr1}
Suppose $(Z,\Sigma_Z)\to(X,\Sigma_X)\to(Y,\sigma)$ is a tower of \'etale 
Galois coverings, and let $C\subseteq\Sigma_Z$ be a conjugacy domain.
We have an exact sequence of groups with operators
$$
1\to\Gal(Z/X)\to\Gal(Z/Y)\to\Gal(X/Y)\to 1.
$$
Then
$$
f_{\exists}\langle Z/X,C\rangle=\langle Z/Y,C\rangle^{\widetilde{}}.
$$
\end{lemma}
Apart from the subtlety that the short exact sequence is needed to deduce that $C$
remains a conjugacy domain with respect to the covering $Z/Y$, the proof
is quite obvious.

\begin{lemma}\label{galtwr2}
Let $(Z,\Sigma_Z)\to(X,\Sigma_X)\to(Y,\sigma)$ be a tower of \'etale Galois
coverings, and assume that $\iota:\Sigma_0\hookrightarrow\Sigma$ is a $\diff$-morphism
satisfying the conditions from \ref{indinj}. Let $(X_0,\Sigma_{X_0})=\iota^*(X,\Sigma)$,
suppose that $(Z_0,\Sigma_{Z_0})$ makes  the square in the diagram 
$$
 \begin{tikzpicture} 
 [cross line/.style={preaction={draw=white, -,
line width=3pt}}]
\matrix(m)[matrix of math nodes, %
row sep=1.5em, column sep=2.2em, text height=1.5ex, text depth=0.25ex]
 { 
          |(z0)|{Z_0} &         |(z)|{Z} \\[1em]
          |(x0)|{X_0} &	       |(x)| {X}  \\[1em]
             |(y)|{Y}      &    		     \\};
\path[->,font=\scriptsize,>=to, thin]
(z0) edge (z) edge (x0) 
(x0) edge node[above]{$(i,\iota)$} (x) edge node[left]{$f_0$} (y)
(z) edge (x)
(x) edge node [right,pos=0.6]{$f$} (y)
;
\end{tikzpicture} 
$$
Cartesian, and let $C_0$ be a conjugacy domain in $\Sigma_{Z_0}$.
Then 
$$
f_{0{\exists}}\langle Z_0/X_0,C_0\rangle=
f_{\exists}i_{\exists}\langle Z_0/X_0,C_0\rangle=f_{\exists}\langle Z/X,\iota_*C_0\rangle,
$$
where $\iota_*C_0$ is the smallest conjugacy domain in $\Sigma_Z$
containing $\iota(C_0)$.
\end{lemma}
\begin{proof}
Note that an $f$-fibre of $i_{\exists}\langle Z_0/X_0,C_0\rangle$ is one of
the blocks
of an $f$-fibre of $\langle Z/X,\iota_*C_0\rangle$ which are in bijective correspondence with the different conjugates of $C_0$ contained in $\iota_*C_0$.
Thus the $f_{\exists}$-images
are equal.
\end{proof}

The main result of this paper is that the class of Galois formulae over fields with
Frobenii is closed under
taking images by $f_\exists$. More precisely, we have the following.

\begin{theorem}\label{eximageisgal}
Let $(S,\sigma)$ be a difference scheme of finite $\sigma$-type over $\Z$ and let 
$f:(X,\sigma)\to (Y,\sigma)$ be a morphism of %
$(S,\sigma)$-difference schemes of
finite $\sigma$-type. 
For every Galois formula $\cA$ of $X$, the subassignment 
$f_{\exists}\cA$ is $\equiv_S$-equivalent to a Galois formula on $Y$, i.e.,
there exists a  Galois formula
$\cB$ on $Y$ such that
$$f_{\exists}\cA\equiv_S\cB.$$ 
When $\cA$ is effectively given, a $\dag$-primitive recursive procedure yields 
an effectively given $\cB$ and a constant $m>0$ 
such that for each closed $s\in S$ and each finite field $k$ with $|k|\geq m$ with $(\bar{k},\varphi_k)$ extending $(\kk(s),\sigma^s)$, 
$$(f_{\exists}\cA)_s(\bar{k},\varphi_k)=f_s(\cA_s(\bar{k},\varphi_k))=\cB_s(\bar{k},\varphi_k).$$
\end{theorem}

Proof.
The proof is by devissage, whereby in each step we calculate the direct image
on a dense open piece and postpone the calculation on the complement to the
next step. At the end of the procedure, we will have obtained the image of each
piece of the domain as a Galois stratification supported on a locally closed piece
of the codomain. To finish, we extend all of these trivially to produce Galois
formulae on the whole of $Y$,
and we take their disjunction to represent the total image as a Galois formula.  

It is straightforward to reduce to the case where $X$ and $Y$ are integral, using
\ref{E1}.
Moreover, the case when $f$ is purely inseparable or purely $\sigma$-inseparable is easily resolved.
 
Thus, by a noetherian induction trick using generic $\sigma$-smoothness \ref{loc-si-smooth} and \ref{chevalley}, after a possible refinement of 
$\cA$, we obtain stratifications $X_i$ and $Y_j$ into integral normal locally 
closed $(S,\sigma)$-subschemes of $X$ and $Y$ such that for every $i$ there
exists a $j$ with $f(X_i)\subseteq Y_j$ and 
$f_i:=f\restriction_{X_i}:(X_i,\sigma)\to (Y_j,\sigma)$ is $\sigma$-smooth.
This can be done in a $\dag$-primitive recursive way, using \ref{E5} and the
effective case of \ref{chevalley}.

By the philosophy of the proof, we can restrict our attention to one of the $f_i$, so we disregard the index $i$ and write $f:(X,\sigma)\to(Y,\sigma)$ in place of $f_i$, and we may assume that $\cA$ on
$X$ is basic, $\cA=\langle Z/X,C\rangle$, where 
$(Z,\Sigma_Z)/(X,\sigma)$ is a connected Galois covering with group $(G,\tilde{\Sigma}_Z)$ and
$C$ is  a $G$-conjugacy domain in $\Sigma_Z$.

By considering the normalisation $\tilde{Y}$ in the relative algebraic closure of 
$\kk(Y)$ inside
$\kk(X)$, we obtain a baby Stein factorisation 
$(X,\sigma)\to (\tilde{Y},\sigma)\to (Y,\sigma)$, 
where the first morphism has generically geometrically integral fibres, and the second is generically $\sigma$-\'etale. By a further localisation, using  \ref{loc-si-integralfibres}
and \ref{loc-si-smooth} and \ref{chevalley} 
we can assume the first morphism has connected fibres and the second is 
$\sigma$-\'etale. All of this is $\dag$-primitive recursive by \ref{E4}, \ref{E2}, \ref{E5}.
Thus, we can split our considerations into two cases.

\noindent{\bf Case 1:} $f$ has geometrically integral fibres. 

The proof of this case can be extrapolated from \cite[\ref{adjpushpull}]{ive-tch} 
for finite-dimensional schemes, but now we have to treat a general case. 
We give a direct proof, following analogous signposts.

Let $(W,\Sigma_W)$ be the normalisation of $(Y,\sigma)$ in the relative algebraic closure of $\kk(Y)$ in $(\kk(Z),\Sigma_Z)$, in the effective case calculated
by \ref{E2} and \ref{E4}.
Then $(W,\Sigma_W)$ is a Galois cover of $(Y,\sigma)$. %
Writing $(X_W,\Sigma_{X_W})=(X,\sigma)\times_{(Y,\sigma)}(W,\Sigma_W)$,  we obtain an exact sequence
\begin{equation}\label{exseq}
1\to\Gal(Z/X_W)\to\Gal(Z/X)\to\Gal(W/Y)\to 1,
\end{equation}
together with a $\diff$-quotient morphism 
$$
\pi:\Sigma_Z\to\Sigma_Z/\Gal(Z/X_W)=\Sigma_W.
$$
Let $D=\pi_*(C)$ be the image of $C$ in $\Sigma_W$, computed by \ref{E7}, and we claim that
$$
f_{\exists}\langle Z/X,C\rangle=\langle W/Y,D\rangle^{\widetilde{}},
$$
i.e., for
all closed $s\in S$, all large enough $k$ with $(\bar{k},\varphi_k)$ extending
$(\kk(s),\sigma^s)$,  
\begin{equation}\label{projeq}
\begin{split}
\{y\in Y_s(\bar{k},\varphi_k) \,\,|\,\, \exists x\in X_s(\bar{k},\varphi_k),  {\varphi}_{k,x}\in{C}, f_s(x)=y\}\\
=\{y\in Y_s(\bar{k},\varphi_k) \,\,|\,\, {\varphi}_{k,y}\in {D}\}.
\end{split}
\end{equation}
A routine verification of the left to right inclusion needs no assumptions on the size of $k$.
Conversely, let $\bar{y}\in Y_s(\bar{k},\varphi_k)$, 
${\varphi}_{k,\bar{y}}=D_0\subseteq{D}$. 
Pick some
$\bar{y}'$ in the fibre of $W/Y$ above $\bar{y}$ with $\varphi_{k,\bar{y}'}\in D_0$.
Let us denote by $y\in Y_s$ and $y'\in W_s$ the loci of $\bar{y}$ and $\bar{y}'$,
$\tilde{y}=\spec^{\sigma^y}(\kk(y))$, $\tilde{y}'=\spec^{\Sigma^{y'}}(\kk(y'))$
(where $\Sigma^{y'}$ is shorthand for $\Sigma_W^{y'}$)
 and consider
the diagram
$$
 
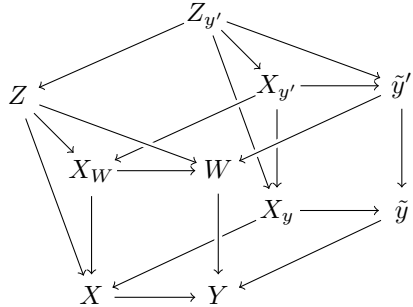
\begin{tikzpicture} 
 [cross line/.style={preaction={draw=white, -,
line width=3pt}}]
\matrix(m)[matrix of math nodes, minimum size=1.7em, nodes={circle},
inner sep=0pt,
row sep=1em, column sep=1em, text height=1.5ex, text depth=0.25ex]
 { 
&&[2em]&[-3.2em] 	 |(u3)|{Z_{y'}} &  & [2em] \\
&&&                       & |(uP)|{X_{y'}} &    |(u2)| {\tilde{y}'}      \\[2em]
&&&                       &  |(u1)|{X_y} &    |(uh)|{\tilde{y}}           \\  [-7em]
 	 |(3)|{Z}  & &&&&\\
                       & |(P)|{X_W} & |(2)| {W}          &&&&\\[2em]
                        &|(1)|{X}                         & |(h)|{Y} &&&&\\};
\path[->,font=\scriptsize,>=to, thin]
(uP) edge (u1) edge (u2)
(u3) edge (u1) edge (u2) edge (uP)
(u2) edge (uh)
(P) edge  (1) edge (2)
(1) edge  (h)
(u3) edge (3)
(uP) edge[cross line] (P)
(u1) edge (1)
(u2) edge[cross line] %
(2)
(uh) edge %
(h)
(u1) edge[cross line] (uh)
(2) edge[cross line]  (h)
(3) edge (1) edge[cross line] (2) edge (P)
;
\end{tikzpicture}
$$
where 
$(X_y,\sigma_y)=(X,\sigma)\times_{(Y,\sigma)}(\tilde{y},\sigma^y)$ is
the fibre of $X$ above $y$, 
$(X_{y'},\Sigma_{y'})=
(X_W,\Sigma_{X_W})\times_{(W,\Sigma_W)}(\tilde{y}',\Sigma^{y'})$ 
is the fibre of $X_W$ above  $y'$, and  
$$(Z_{y'},\Sigma_{y'})=
(Z,\Sigma_Z)\times_{(X_W,\Sigma_{W_X})}(X_{y'},\Sigma_{y'})
=(Z,\Sigma_Z)\times_{(W,\Sigma_W)}(\tilde{y}',\Sigma^{y'})$$
is the fibre of $Z$ above $y'$.
By construction, $Z\to W$ has geometrically connected fibres so we conclude that
$Z_{y'}$ is geometrically connected and $\Gal(Z_{y'}/X_{y'})\cong\Gal(Z/X\times_YW)$.
In the diagram with exact rows
\begin{center}\begin{tikzpicture}
\matrix(m)[matrix of math nodes, row sep=2.4em, column sep=1.5em, text height=1.5ex, text depth=0.25ex]
{
|(l1)|{1} &|(l2)|{\Gal(Z_{y'}/X_{y'})}& |(l3)|{{\Gal}(Z_{y'}/X_y)} &|(l4)|{{\Gal}(\tilde{y}'/\tilde{y})}&|(l5)|{1}\\
|(1)|{1} & |(2)| {\Gal(Z/X_W)}& |(3)| {{\Gal}(Z/X)} & |(4)|{{\Gal}(W/Y)} &|(5)|{1}
\\};
\path[->,font=\scriptsize,>=to, thin]
(1) edge (2)
(2) edge (3)
(3) edge (4)
(4) edge (5)
(l1)edge(l2)
(l2)edge(l3)
(l3)edge(l4)
(l4)edge(l5)
(l2)edge(2)
(l3)edge(3)
(l4)edge(4);
\end{tikzpicture}\end{center}
the left vertical arrow is an isomorphism and 
${\Gal}(\tilde{y}'/\tilde{y})={\Gal}(\kk(y')/\kk(y))$.
It follows that $\Gal(Z_{y'}/X_y)=\Gal(Z/X)\times_{\Gal(W/Y)}\Gal(\tilde{y}'/\tilde{y})$
and we get a Cartesian diagram of difference structures:
$$
\begin{tikzpicture} 
\matrix(m)[matrix of math nodes, row sep=1.5em, column sep=1.4em,text height=1.3ex, text depth=0.25ex]
 {                       & |(P)|{\Sigma_{\smash{\mathrlap{y'}{}}}} &           \\
 |(1)|{\Sigma_{\smash{\mathrlap{Z}{}}}} &                         & |(2)| {\Sigma^{\smash{\mathrlap{y'}{}}}}\\
                & |(h)|{\Sigma_{\smash{\mathrlap{W}{}}}}            &\\}; 
\path[->,font=\scriptsize,>=to, thin]
(P) edge  (1) edge (2)
(1) edge  (h)
(2) edge  (h);
\end{tikzpicture}
$$
Thus we can find a conjugacy domain $C'\subseteq\Sigma_{y'}$ which maps into 
$C$ in
$\Sigma_Z$ (and eventually to $D_0\subseteq\Sigma_W)$, as well as onto 
$\varphi_k$.
It
suffices to find an $x\in X_y(\bar{k},\varphi_k)$ with ${\varphi}_{k,x}\subseteq C'$ 
with respect to the
cover $Z_{y'}/X_y$, and this is possible for large enough $k$ by Twisted Chebotarev 
\cite[\ref{twisted-cebotarev}]{ive-tch}. The relevant bound for the size of $k$ can be
calculated by \ref{E8}.  

\noindent{\bf Case 2:} $f$ is $\sigma$-\'etale. 
Using \ref{gal-cl}, modulo a $\sigma$-localisation, we may assume that
the quasi-Galois closure $h:(\tilde{X},\tilde{\Sigma})\to (Y,\sigma)$ of $X$ over $Y$ strictly 
dominates $f:(X,\sigma)\to(Y,\sigma)$, i.e., that the morphism 
$r:(\tilde{X},\tilde{\Sigma})\to (X,\sigma)$ is surjective. Then 
$$
f_{\exists}\cA=f_{\exists}r_{\exists}r^*\cA=h_{\exists}r^*\cA,
$$
so it is enough to show that the direct image by $h_{\exists}$ of a Galois formula
is again Galois. In the effective case, we can do this via  \ref{E6}.

In other words, we may assume that $f:(X,\Sigma)\to(Y,\sigma)$ is $\sigma$-\'etale
quasi-Galois, so we can benefit from Babbitt's decomposition. Indeed, modulo
a localisation, \ref{babb-sch} yields a decomposition of $f$ as
$$
(X,\Sigma)\simeq(X_n,\Sigma_n)\to\cdots\to(X_1,\Sigma_1)\to(X_0,\Sigma_0)\to(Y,\sigma),
$$
with $(X_0,\Sigma_0)\to(Y,\sigma)$ finite \'etale quasi-Galois, and for $i\geq 0$, 
$(X_{i+1},\Sigma_{i+1})\to(X_i,\Sigma_i)$ \'etale benign  Galois. This can be achieved
in a $\dag$-primitive recursive way by using \ref{E3} and \ref{E4}.
We can reduce to two
 subcases as follows.

\noindent{\bf Case 2(a):} $f_0:(X_0,\Sigma_0)\to(Y,\sigma)$ is finite \'etale quasi-Galois. %
We are given a Galois  covering $(Z_0,\Sigma_{Z_0})\to (X,\Sigma_0)$ and
a conjugacy class $C_0$ in $\Sigma_{Z_0}$. Thus 
$(Z_0,\Sigma_{Z_0})\to(Y,\sigma)$ is quasi-Galois so it is subsumed in 
a Galois covering $(Z,\Sigma_Z)\to(Y,\sigma)$ which is obtained by 
essentially just expanding the difference structure via 
$\iota_Z:\Sigma_{Z_0}\to\Sigma_Z$ so that $Z_0=\iota_Z^*Z$ and we can
apply \ref{galtwr2} and \ref{galtwr1} to obtain
$$
f_{0{\exists}}\langle Z_0/X_0,C_0\rangle=
f_{\exists}\langle Z/X,\iota_*C_0\rangle=
\langle Z/Y,\iota_*C_0\rangle.
$$
The relevant calculations in the effective case are performed using \ref{E6} and
\ref{E7}.

\noindent{\bf Case 2(b):} $f:(X,\Sigma)\to(Y,T)$ is benign \'etale quasi-Galois.  
We are given a Galois covering $(Z,\Sigma_Z)\to(X,\Sigma)$ and a conjugacy domain $C$ in $\Sigma_Z$. Babbitt's decomposition \ref{babb-sch} applied to $Z/Y$ yields 
a sequence 
$$
(Z,\Sigma_Z)\simeq (Z_n,\Sigma_n)\to\cdots\to (Z_1,\Sigma_1)\to 
(Z_0,\Sigma_0)=(W,\Sigma_W)\to (Y,T)$$
 where $(W,\Sigma_W)/(Y,T)$  can be assumed to be  
a finite \'etale Galois covering %
and $Z_{i+1}/Z_{i}$ is benign for $i\geq 0$. Since $\kk(X)$ is linearly disjoint from $\kk(W)$ over $\kk(Y)$,
we obtain an exact sequence of the form (\ref{exseq}) again, and we have
the corresponding $\diff$-morphism 
$\pi:\Sigma_Z\to\Sigma_Z/\Gal(Z/X_W)=\Sigma_W$. 
Let $D=\pi_*C$ be the image of $C$ in $\Sigma_W$,
and we claim that (\ref{projeq}) holds for any $s$ closed in $S$ and $k$ and  
$(\bar{k},\varphi_k)$ extending $(\kk(s),\varphi_s)$. 
To see the non-trivial inclusion, let 
$y$ be an element of the right hand side and let $z_0\in W=Z_0$ such that $z_0\mapsto y$ and
${\varphi}_{k,z_0}\in {D}$. Using the property \ref{benignlift} repeatedly, we can lift $z_0$ through
the `stack' of benign extensions $Z_{i+1}/Z_{i}$ to a point 
$z\in Z^{\tilde{\sigma}}(\bar{k},\varphi_k)$ with $\tilde{\sigma}\in{C}$, and then the image
$x$ of $z$ in $X_s$ has the properties $\varphi_{k,x}\sim\tilde{\sigma}\in C$ 
and $f(x)=y$.
This case is $\dag$-primitive recursive by \ref{E2}, \ref{E4}, \ref{E7}.

\begin{corollary}\label{unimageisgal}
With assumptions of \ref{eximageisgal}, it makes sense to define a subassignment
$$
f_{\forall}\cA=\lnot f_{\exists}(\lnot\cA),
$$
and it is again a Galois formula on $Y$.
\end{corollary} 

\subsection{Quantifier elimination for Galois formulae}

Let $(R,\sigma)$ be an integral normal
 difference ring of finite $\sigma$-type over $\Z$,
and let $(S,\sigma)=\spec^\sigma(R)$.
\begin{definition}
\begin{enumerate}
\item
A \emph{first-order formula} over $(S,\sigma)$ is a first-order expression
built in the usual way starting from terms which are difference polynomials
with coefficients in $(R,\sigma)$. If $x_1,\ldots,x_n$ are the free variables of a
formula $\theta$, and $r_1,\ldots,r_m\in R$ are all the coefficients of all polynomials
appearing as terms of $\theta$, we can express this dependence by
writing 
$$\theta(x_1,\ldots,x_n;r_1,\ldots,r_n),$$
where the $r_i$ are thought of as \emph{parameters} of $\theta$. 
\item An $(R,\sigma)$-formula $\theta(x_1,\ldots,x_n;r_1,\ldots,r_n)$ gives rise to
a subassignment $\tilde{\theta}$ of $\Af^n_{(S,\sigma)}$ by the following 
procedure. Let $s\in S$, and let $(F,\varphi)$ be an algebraically closed difference
field extending $(\kk(s),\varphi^s)$. Taking the images of the $r_i$ by the
composite
$$
(R,\sigma)\to(\OO_{S,s},\sigma_s^\sharp)\to(\kk(s),\sigma^s)\to (F,\varphi),
$$
we obtain a honest first-order formula in the language of difference rings on
the field $(F,\varphi)$, and we take its set of realisations to be the value
$$\tilde{\theta}(s,(F,\varphi))\subseteq\Af^n_s(F,\varphi).$$
\item An $(S,\sigma)$-subassignment $\cF$ of $\Af^n_S$ is called \emph{definable} if
there exists a first-order formula $\theta(x_1,\ldots,x_n)$ over $(R,\sigma)$
such that $\cF=\tilde{\theta}$.
\end{enumerate}
\end{definition}

\begin{theorem}[Quantifier elimination for fields with Frobenii]\label{qe-ffrob}
The class of definable $(S,\sigma)$-subassignments is equal to the class of 
$(S,\sigma)$-Galois formulae modulo the relation $\equiv_S$, i.e., 
with respect to fields with Frobenii over $S$. The quantifier elimination procedure
is $\dag$-primitive recursive.
\end{theorem}

\begin{proof}
Let us show by induction on the complexity of a first-order formula that
every $(S,\sigma)$-formula in the language of rings $\theta(x_1,\ldots,x_n)$
is equivalent to a Galois formula on $\Af^n_{(S,\sigma)}$.
\begin{enumerate}
\item If $\theta(x_1,\ldots,x_n)$ is a positive atomic formula, it is given by
a difference-polynomial equation $P(x_1,\ldots,x_n)=0$, which cuts out
a closed difference subscheme $Z$ of  $\Af^n_{(S,\sigma)}$. We can stratify
the affine space into normal locally closed pieces $X_i$ such that each
piece is either completely in $Z$ or in its complement. For each $X_i$, we
choose a trivial Galois covering $(X_i,\sigma)\to(X_i,\sigma)$, and we let 
$C_i=\{\sigma\}$ when
$X_i\subseteq Z$, and $C_i=\emptyset$ otherwise. Then 
$\cA=\langle  \Af^n_S,X_i/X_i,C_i\rangle$ 
has the property that
$$\tilde{\theta}=\tilde{\cA}.$$
\item If $\theta(\bar{x})=\theta_1(\bar{x}_1)\land\theta_2(\bar{x}_2)$, 
where it is assumed that $\bar{x}$ is the union of variables in $\bar{x}_1$
and $\bar{x}_2$, we choose 
the corresponding projections $p_i:\Af^{|\bar{x}|}\to\Af^{|\bar{x}_i|}$. By induction
hypothesis, we can find Galois formulae $\cA_i$ on $\Af^{|\bar{x}_i|}$
such that $\theta_i\equiv_S\cA_i$.
Then
$$
\theta\equiv_S p_1^*\cA_1\land p_2^*\cA_2.
$$
\item If $\theta=\theta_1\lor\theta_2$, we proceed analogously to the previous step.
\item If $\theta=\lnot\theta'$, and $\theta'\equiv_S\cA$, then
$$\theta\equiv_S\lnot\cA.$$
\item\label{ppet} If $\theta(x_2,\ldots,x_n)=\exists x_1 \theta'(x_1,x_2,\ldots,x_n)$, and
$\theta'\equiv_S\cA$ on $\Af^n$, writing $x_1$ for the projection $\Af^n\to\Af^{n-1}$
to the variables $x_2,\ldots,x_n$, we have that
$$
\theta\equiv_S\exists x_1\theta'\equiv_S  {x_1}_{\exists}\cA,
$$
which is Galois by \ref{eximageisgal}.
\item\label{ssest} If $\theta=\forall x_1\theta'$, and $\theta'\equiv_S\cA$, then
$$
\theta\equiv_S\forall x_1\theta'\equiv_S {x_1}_{\forall}\cA,
$$
which is Galois by \ref{unimageisgal}.
\end{enumerate}
We have checked all cases so the induction is complete. Note that
working over fields with Frobenii is only crucial in steps \ref{ppet} and \ref{ssest}.

Conversely, suppose we have a Galois stratification 
$\cA=\langle \Af^n_{(S,\sigma)}, Z_i/X_i,C_i\rangle$. By refining it further, we
may assume that each Galois covering $(Z_i,\Sigma_i)\to (X_i,\sigma)$ with
group $(G,\tilde{\Sigma})$ is embedded in some
affine space, in the sense that $Z_i$ is embedded in some $\Af^m_S$,
and all automorphisms corresponding to elements of $G$ are restrictions of difference rational endomorphisms of $\Af^m_S$ to $Z_i$, and the canonical projection
$Z_i\to X_i$ is a restriction of difference rational morphism $\Af^m_S\to\Af^n_S$.
Then, if $C_i$ is the conjugacy class of some element $\sigma_i\in \Sigma$, the set
$$
\{x\in X_i: \ar(x)\subseteq C_i\}=\{x\in X_i: \exists z\in Z_i^{\sigma_i}, z\mapsto x\}
$$  
is clearly expressible in a first-order way using an existential formula in the
language of difference rings. When $C_i$ is a union of conjugacy classes, we
take the disjunction of the corresponding difference ring formulae.
\end{proof}

Let $T_\infty$ be the set of first-order sentences true in difference fields
$(\bar{k},\varphi_k)$ with $k$ a sufficiently large finite field.

\begin{corollary}
 The theory $T_\infty$ is
decidable by a $\dag$-primitive recursive procedure. Moreover, for
each first-order sentence $\theta\in T_\infty$ a $\dag$-primitive recursive procedure
can compute the (finite) list of exceptional finite fields $k$ such that
$\theta$ does not hold in $(\bar{k},\varphi_k)$.
\end{corollary}
\begin{proof}
The quantifier elimination procedure produces a Galois stratification $\cA$ on 
the base $S=\spec(\Z)$ and a constant $m$ such that for every $p\in S$, 
and every $k$ of characteristic $p$ with $|k|\geq m$, 
$\theta(\bar{k},\varphi_k)=\cA(\bar{k},\varphi_k)$. The stratification $\cA$
stipulates the existence of a localisation
$S'=\Z[1/N]$ of $S$, a Galois cover $Z/S'$ and a conjugacy class $C$ in 
$\Gal(Z/S')$ such that, for $p\in S'$ (i.e., for $p$ not dividing $N$),
and $k$ of characteristic $p$ with $|k|\geq m$, $\theta$ holds in $(\bar{k},\varphi_k)$
if and only if $\varphi_k\in C$. By (the classical) Chebotarev's density theorem, this can hold for all but finitely many $p$ if and only if $C=\Gal(Z/S')$, which can be effectively
checked by \ref{E7}.

For each field $(\bar{k},\varphi_k)$ with characteristic of $k$ dividing $N$, or
$|k|<m$, once we interpret $\sigma$ as the Frobenius $\varphi_k$ with 
$\varphi_k(\alpha)=\alpha^{|k|}$, the formula $\theta$ can be treated as a formula
in the language of rings, which can be decided  by the
well-known primitive recursive decision procedure for the algebraically closed
 field $\bar{k}$.
\end{proof}

A more model-theoretic restatement of the above theorem would say
that the theory  $T_\infty$ of fields with Frobenii allows quantifier elimination down to
the class of Galois formulae. Given that 
$T_\infty$ happens to be (\cite{udi}) the theory of existentially closed difference
fields (ACFA), let us state an appropriate analogue of the above result.

We must emphasise that the statement below can be obtained unconditionally, i.e., 
without appealing to \cite{udi}, by replacing the use of \cite[\ref{udiLW}]{ive-tch} in the
present paper by the use of existential-closedness of models of ACFA
(i.e., by the use of the `ACFA-axiom'). This will be done in a separate paper
\cite{ive-tgsacfa}.

\begin{theorem}\label{stratACFA}
Let $(k,\sigma)$ be a difference field. 
Let $\psi(x)=\psi(x;s)$ be a first order formula in the language of difference rings in variables 
$x=x_1,\ldots,x_n$ with parameters $s$ from $k$. There exists a Galois 
stratification $\cA$ of the difference affine $n$-space over $k$ such that for 
every model $(F,\varphi)$ of ACFA which extends $(k,\sigma)$,
$$
\psi(F,\varphi)=\cA(F,\varphi).
$$
\end{theorem}

\bibliographystyle{plain}

\end{document}